\pdfoutput=1
\documentclass{article} 
\usepackage[dvipsnames]{xcolor}
\usepackage{hyperref}
\usepackage{url}
\usepackage{smile}
\usepackage{graphicx}
\usepackage{algorithm}
\usepackage{algorithmic}
\usepackage[margin=1.0in]{geometry}
\usepackage[export]{adjustbox}
\usepackage{mathtools, cuted}
\usepackage{bbm}
\usepackage{wrapfig}
\usepackage{subcaption}
\usepackage{caption}
\usepackage{enumitem}
\usepackage{tabularx}
\usepackage{smile}
\usepackage{tcolorbox}
\usepackage{enumitem}
\usepackage{mathtools, cuted}
\usepackage{caption} 
\usepackage{subcaption}

\numberwithin{equation}{section}

\renewcommand{\frac}{\tfrac}

\usepackage{float}
\newfloat{procedure}{tbhp}{lop}
\floatname{procedure}{Procedure}

\usepackage{footnote}
\makesavenoteenv{algorithmic}

\usepackage[nocompress]{cite}

\hypersetup{
    colorlinks=true,
    linkcolor=blue,
    anchorcolor=blue, 
    citecolor=blue,
}

\allowdisplaybreaks

\title{
A Novel Catalyst Scheme for Stochastic Minimax Optimization
    }
\author{
    Guanghui Lan \thanks{H. Milton Stewart School of Industrial and Systems Engineering, Georgia Institute of Technology, Atlanta, GA, 30332. (E-mail: \url{george.lan@isye.gatech.edu}).}
    \and 
        Yan Li   \thanks{H. Milton Stewart School of Industrial and Systems Engineering, Georgia Institute of Technology, Atlanta, GA, 30332. (E-mail: \url{yli939@gatech.edu}).}
}
\date{\vspace{-5ex}}

\begin{document}
{
\makeatletter
\addtocounter{footnote}{1} 
\renewcommand\thefootnote{\@fnsymbol\c@footnote}%
\makeatother
\maketitle
}

\maketitle

\begin{abstract}
This paper presents a proximal-point-based catalyst scheme for simple first-order methods applied to convex minimization and convex-concave minimax problems. In particular, for smooth and (strongly)-convex minimization problems, the proposed catalyst scheme, instantiated with a simple variant of stochastic gradient method, attains the optimal rate of convergence in terms of both deterministic and stochastic errors. For smooth and strongly-convex-strongly-concave minimax problems, the catalyst scheme attains the optimal rate of convergence for deterministic and stochastic errors up to a logarithmic factor. To the best of our knowledge, this reported convergence seems to be attained for the first time by stochastic first-order methods in the literature. We obtain this result by designing and catalyzing a novel variant of stochastic extragradient method for solving smooth and strongly-monotone variational inequality, which may be of independent interest.
\end{abstract}


\section{Introduction}\label{sec_intro}

We consider the problem of 
\begin{align}\label{def_problem}
\textstyle
 \min_{x \in X} \big\{ f(x):= \max_{y \in Y} F(x,y) \big\},
\end{align}
where $F: X \times Y \to \RR$ is $\mu_p$-strongly-convex w.r.t. $x$ and $\mu_d$-strongly-concave w.r.t. $y$ for some $\mu_p, \mu_d \geq 0$.
That is,
\begin{align}
F(x_1, y) - F(x_2, y) - \langle \nabla_x F(x_2, y), x_1-x_2 \rangle \ge \tfrac{\mu_p}{2} \|x_1 - x_2\|^2, \ \forall x_1, x_2 \in X, \label{sc_x}\\ 
F(x, y_1) - F(x, y_2) - \langle \nabla_y F(x, y_2), y_1-y_2 \rangle \le -\tfrac{\mu_d}{2} \|y_1 - y_2\|^2, \ \forall y_1, y_2 \in Y. \label{sc_y}
\end{align}
In addition,  $X$ and $Y$ are compact convex sets. 
In view of the strong duality, without loss of generality, going forward  we assume that $\mu_d \geq \mu_p$. 
We assume additionally that the gradient of $F$ is $L$-Lipschitz.
That is,
$
\norm{\nabla F(x,y) - \nabla F(x', y')}_* \leq L \norm{(x, y) - (x', y')}
$
for some $L > 0$,\footnote{
It should be noted that some prior development (e.g., \cite{wang2020improved, jin2022sharper, zhao2022accelerated, zhang2021robust}) consider using $L_x$, $L_y$, and $L_{xy}$ to model potentially different smoothness along primal and dual variables and the strength of their  non-linear coupling.
In this manuscript we take $L = \max \cbr{L_x, L_y, L_{xy}}$ and do not differentiate these three parameters when defining the problem class \eqref{def_problem}.
}
where $\norm{\cdot}_*$ denotes the dual norm of $\norm{\cdot}$.
For the discussions in this manuscript, we will consider both the deterministic setting, where exact first-order information of $F$ is available, and the stochastic setting, where we only have access to a stochastic oracle producing unbiased estimate of the gradient information. 
Clearly, \eqref{def_problem} subsumes the problem of convex optimization when $Y$ is a singleton. 

To solve \eqref{def_problem}, it is natural to consider the variational inequality (VI) associated with its optimality condition.
For general convex-concave problems where $\mu_p = \mu_d = 0$, \cite{nemirovski2004prox} presents the mirror-prox method, which can be viewed as a generalization to the extragradient method \cite{korpelevich1976extragradient}, and exhibits an $\cO(1/\epsilon)$ iteration complexity in the deterministic setting for obtaining a duality gap upper bounded by $\epsilon$.
Stochastic variants of mirror-prox has also been discussed in \cite{juditsky2011solving, chen2017accelerated} with an $\cO(1/\epsilon^2)$ sample complexity.
For strongly-convex-strongly-concave problems where $\mu \coloneqq \min\cbr{\mu_p, \mu_d} > 0$, the corresponding VI becomes strongly-monotone.
In the deterministic setting, to output an approximate solution that has its distance to the optimal solution bounded by $\epsilon$, an  optimal $\cO({L}/{\mu} \log ({1}/{\epsilon}))$ iteration complexity is first obtained by the dual extrapolation method \cite{nesterov2007dual, nesterov2006solving}. 
By adopting stochastic approximation method \cite{nemirovski2009robust, lan2020first} to the strongly monotone VI, an $\cO(1/\epsilon)$ sample complexity can be obtained in the stochastic setting. 
Algorithmic development of VI has been an active line of research \cite{yousefian2017smoothing, iusem2017extragradient, cui2021analysis, monteiro2010complexity, dang2015convergence, malitsky2015projected, beznosikov2021distributed, mokhtari2020unified, popov1980modification, tseng1995linear}, partly due to its emerging role in data science.
Until very recently, a method that obtains optimal complexities in both deterministic and stochastic settings is proposed in \cite{kotsalis2022simple}.
In addition to VI-based methods, another major class of primal-dual methods consider solving a special case of \eqref{def_problem} with bilinear coupling between primal and dual variables. 
Methods for solving these bilinear minimax problem include Nesterov's smoothing scheme \cite{nesterov2005smooth},  primal-dual hybrid gradient method \cite{chambolle2011first, chambolle2016ergodic}, and more recent development with optimal complexities    \cite{jin2022sharper, chen2014optimal}. 

While the adoption of the VI perspective can be convenient, the direct application of the aforementioned class of VI-based methods for \eqref{def_problem} is known to be non-optimal even in the deterministic setting,  when \eqref{def_problem} becomes asymmetric in the sense that $\mu_p \neq \mu_d$.
In particular, it can be seen that when $\mu_p = 0$, the corresponding VI is only monotone, which implies an $\cO(L/\epsilon)$ iteration complexity by VI-based methods. 
This should be contrasted with the lower bound of $\Omega(L/\sqrt{\mu_d \epsilon})$ recently established in \cite{ouyang2021lower} for the class of bilinear convex-concave problems. 
Additionally, for strongly-convex-strong-concave problem with $\mu_d > \mu_p$, the obtained $\cO({L}/{\mu} \log ({1}/{\epsilon}))$ iteration complexity by VI-based methods is strictly larger than the existing lower bound $\Omega(L/\sqrt{\mu_p \mu_d} \log(1/\epsilon))$ presented in \cite{zhang2019lower, ibrahim2019lower}. 
Such a gap between lower and upper bounds is recently closed in \cite{lin2020near} and later improved in \cite{wang2020improved}, where an inexact accelerated proximal point framework is proposed and obtain the aforementioned lower complexity bounds up to logarithmic factors. 
It should be noted that there has been a rich line of research in developing inexact (accelerated) proximal point methods for convex optimization, and more generally, monotone inclusion problems \cite{guler1992new, salzo2012inexact, monteiro2013accelerated, monteiro2010complexity}.  

Nevertheless, a limitation of existing development for solving problem \eqref{def_problem} is the absence of a unified algorithm 
that can simultaneously attain optimal complexities in both deterministic and stochastic regimes. 
Indeed, the aforementioned proximal-point based optimal methods \cite{lin2020near, wang2020improved} in the deterministic setting seem difficult to generalize to the stochastic setting in view its somewhat stringent error condition for each proximal step. 
On the other hand, existing methods obtaining optimal reduction of stochastic error \cite{zhao2022accelerated, zhang2021robust, juditsky2011solving, chen2017accelerated}  suffer from non-optimal reduction in its deterministic error.
Our essential objective in this manuscript is to design a single method that, with proper parameter specifications, obtain optimal iteration complexity in the deterministic setting, and optimal reduction of both deterministic and stochastic errors in the stochastic setting. 
Such a development bears the potential of several practical benefits.
For instance, one can reduce the iteration complexity in the stochastic setting by using mini-batches; 
improve the communication complexity in the distributed or federated settings;
and take advantage of variance reduction techniques available for smooth and finite-sum problems. 

Our main contributions can be summarized into the following aspects. 
First, we propose a novel inexact accelerated proximal point framework for both convex optimization and minimax optimization. 
The proposed framework is referred to as a catalyst scheme as it is capable of accelerating simple non-optimal methods and subsequently  obtaining optimal complexities. 
In a nutshell, within the catalytic process, the to-be-catalyzed non-optimal methods are used by the proposed framework as a subroutine for solving the inexact proximal step. 
 Compared to the inexact accelerated proximal point framework considered in \cite{lin2020near}, the proposed catalyst framework properly modifies the search sequences within the acceleration scheme using the outputs of the inexact proximal update,  leading to its optimal performances in both deterministic and stochastic regimes. 

Second, we show that for smooth and convex problems, by catalyzing a simple variant of (stochastic) gradient descent method, the proposed framework obtains an $\cO(\sqrt{L/\epsilon})$ iteration complexity in the deterministic setting, and an 
$\cO(\sqrt{L/\epsilon} + \sigma^2/\epsilon^2)$ sample complexity in the stochastic setting. 
Similarly, for smooth and strongly-convex problems, the proposed framework can catalyze (stochastic) gradient descent and obtain an  $\cO(\sqrt{L/\mu} \log(1/\epsilon))$ iteration complexity in the deterministic setting, and  
$\cO(\sqrt{L/\mu} \log(1/\epsilon) + \sigma^2 / (\mu \epsilon))$ sample complexity in the stochastic setting. 
Notably, these obtained complexities are optimal for smooth and (strongly)-convex optimization.
In particular, even when restricting to the deterministic setting, the obtained iteration complexity improves over existing catalyst scheme for convex optimization \cite{lin2018catalyst} by removing the additional logarithmic factor.

Third, we apply the proposed catalyst scheme to the minimax problem considered in \eqref{def_problem}. 
En route, we  develop a novel variant of extragradient method, named regularized extragradient (REG), applicable to strongly-monotone VI.
REG method obtains $\cO(L/\mu \log(1/\epsilon))$ iteration complexity in the deterministic setting and 
$\cO(L/\mu \log(1/\epsilon) + \sigma^2/ (\mu \epsilon) ) $ sample complexity in the stochastic setting. 
Notably, REG simultaneously controls the gap value and the distance to the reference point (determined a posteriori), a property that makes it an ideal candidate for the catalyst scheme compared to existing alternatives that only consider  the distance to the solution (e.g., \cite{kotsalis2022simple, mokhtari2020unified, beznosikov2021distributed}). 
We then further catalyze the proposed REG method and establish that, in the deterministic setting, to obtain an $\epsilon$-optimal solution of \eqref{def_problem}, the iteration complexity can be bounded by $\tilde{\cO}(L/\sqrt{\mu_d \epsilon})$ for convex-strongly-concave problems, 
and $\tilde{\cO}(L/\sqrt{\mu_d \mu_p} \log(1/\epsilon))$ for strongly-convex-strongly-concave problems. 
In the stochastic setting, we obtain sample complexities that are optimal up to logarithmic factors. 
Specifically, for convex-strongly-concave problems, we establish an
$
\tilde{\cO} ( L/\sqrt{\mu_d \epsilon} + \sigma^2 / \epsilon^2)
$
sample complexity.
For strongly-convex-strongly-concave problems, we obtain an 
$
\tilde{\cO} (L/\sqrt{\mu_d \mu_p} \log(1/\epsilon) + \sigma^2 / (\mu \epsilon))
$
sample complexity.
To the best of our knowledge, the proposed catalyst scheme seems to be the first method that attains optimal complexities (up to logarithmic factors) in both deterministic and stochastic regimes for smooth and strongly-convex-strongly-concave problems. 

Finally, it should be noted that the catalyst scheme can also be applied to convex-concave problems where $\mu_p = \mu_d = 0$, with a black-box reduction technique considered in \cite{lin2020near}. In a nutshell, one can perturb \eqref{def_problem} with strongly-convex and strongly-concave regularizers while controlling the distance between the  original and the perturbed objectives. 
We do not explicitly delve into this reduction process in detail to maintain the simplicity of our presentation.

The rest of the manuscript in organized as follows. 
In section \ref{sec_opt}, we present the catalyst scheme for convex optimization, and establish its optimal iteration and sample complexities when combined with the stochastic gradient descent method. 
In section \ref{sec_eg}, we propose and analyze a variant of extragradient method for solving smooth and strongly monotone variational inequality, and establish its optimal iteration and sample complexities.
In section \ref{sec_minimax_deterministic}, we present the catalyst scheme for minimax optimization in the deterministic setting, and establish its  iteration complexity when catalyzing the extragradient variant. 
Finally, in section \ref{sec_minimax_stoch} we analyze the catalyst scheme for minimax optimization in the stochastic setting, and establish its  sample complexity when combined with the stochastic extragradient method proposed in Section \ref{sec_eg}.


\section{Catalyst Scheme for Convex Optimization}\label{sec_opt}

In this section, we consider the convex programming given in the form of
\begin{equation} \label{cp}
f^* := \min_{x \in X} f(x),
\end{equation}
where $X \subseteq \RR^d$ is a closed convex set and $f: X \to \RR$ is
a convex function such that
\begin{align} \label{eq:phi_convexity}
f(\gamma x_1 + (1-\gamma) x_2)
\le \gamma f(x_1) + (1-\gamma) f(x_2) - \tfrac{\mu \gamma (1-\gamma)}{2}  \|x_1 - x_2\|^2, ~ \forall x_1, x_2 \in X,
\end{align}
for some $\mu \ge 0$. In particular, we say that the problem is strongly convex
if $\mu > 0$.
In addition, we assume $f$ is smooth with modulus $L$:
\begin{align*}
\abs{ f(x_1) - f(x_2) - \inner{\nabla f(x_1)}{x_2 - x_1} } \leq \frac{L}{2} \norm{x_1 - x_2}^2, ~  ~ \forall x_1, x_2 \in X.
\end{align*}

The general framework of the proposed catalyst scheme is stated in Algorithm \ref{alg:basic_cat}, which can be viewed as a meta algorithm that wraps around an optimization method $\cA$ and improves its computational efficiency. 
At each iteration, for a given prox-center $\hat x_k$ and some $\beta_k \ge 0$, the catalyst scheme considers the following proximal update 
\begin{align} \label{eq:cp_subproblem}
\min_{x \in X} \cbr{ \phi_k(x) := f(x) + \tfrac{\beta_k}{2} \norm{x - \hat x_k}^2 }.
\end{align}
We let $(\tilde x_k, x_k) = \cA(\phi_k)$ denote the output of the to-be-catalyzed method $\cA$ for solving the above proximal update \eqref{eq:cp_subproblem} while initialized at $\hat{x}_k$.
As will be clarified later in our analysis, this particular choice of initial point when calling $\cA(\phi_k)$ is of particular importance. 
In addition, we assume  
\begin{align} \label{eq:cp_inexact}
\EE[\phi_k(\tilde x_k) -\phi_k(\tilde x) + \tfrac{\alpha_k}{2}  \|\tilde x - x_k\|^2] \le  \tfrac{\varepsilon_k}{2}  \|\tilde x - \hat x_k \|^2 + \delta_k, ~ \forall \tilde x \in X, 
\end{align}
for some  $\alpha_k,   \delta_k, \varepsilon_k \ge 0$.
Here, the expectation is taken with respect to the possible randomness introduced when solving problem~\eqref{eq:cp_subproblem}.
It should be noted that the implementation of the catalyst scheme only requires $\cbr{\alpha_k}$.
Clearly, when the proximal step is computed exactly, we have $\alpha_k = \beta_k$, $\varepsilon_k = \delta_k = 0$.

\begin{algorithm}[H]
\caption{\texttt{Catalyst}$(\cA)$:  catalyst scheme for convex optimization}
\begin{algorithmic}
\STATE {\bf Input:} initial points $\overline{x}_0 = \tilde{x}_0$, number of iterations $K > 0$.
\FOR {$k = 1, 2, \ldots, K$}
\STATE{
\vspace{-0.3in}
\begin{align}
& \hat x_k = \gamma_k \bar x_{k-1} + (1-\gamma_k) \tilde x_{k-1}.  \label{eq:define_hat_x} \\
& (\tilde x_k, x_k) =
 \cA(\phi_k) ~  \text{s.t.  \eqref{eq:cp_inexact} holds for some $(\alpha_k, \varepsilon_k, \delta_k)$.}\\
& \bar x_{k} =  
 \tfrac{1}{\alpha_k \gamma_k  +  \mu (1-\gamma_k) }\left[\alpha_k x_k + ( \mu -  \alpha_k) (1-\gamma_k) \tilde x_{k-1}\right]. \label{eq:define_bar_x}
\end{align}
}
\ENDFOR
\STATE{{\bf Output:} $\tilde{x}_K$}
\end{algorithmic} \label{alg:basic_cat}
\end{algorithm}

Concrete method $\cA$ for constructing $(\tilde x_k, x_k)$ satisfying condition \eqref{eq:cp_inexact} will be discussed in Section \ref{subsec_catalyst_sgd}. 
Before that, let us first establish the following generic characterizations for each step of Algorithm \ref{alg:basic_cat}. 

\begin{lemma}\label{convex_each_step_characterization}
Let $\{(\tilde x_k, x_k, \bar x_{k}, \hat x_k)\}$ be a sequence of iterates generated by Algorithm~\ref{alg:basic_cat}. Then
for any $k \ge 1$, we have
\begin{align}
&\EE[f(\tilde x_k) - f(x)  +  \frac{\alpha_k \gamma_k^2 + \gamma_k(1-\gamma_k) \mu}{2}  \| x - \bar x_{k}\|^2] \nonumber\\
\leq & (1-\gamma_k) [f(\tilde x_{k-1}) - f(x)]  + \tfrac{(\beta_k + \varepsilon_k) \gamma_k^2}{2} \|x - \bar x_{k-1} \|^2 +\delta_k, ~ \forall x \in X. \label{eq:main_recursion}
\end{align}
\end{lemma}

\begin{proof}
By \eqref{eq:cp_subproblem} and \eqref{eq:cp_inexact}, we have
\begin{align*}
\EE[f(\tilde x_k) + \tfrac{\beta_k}{2} \|\tilde x_k - \hat x_k\|^2 +\tfrac{\alpha_k}{2} \|\tilde x - x_k\|^2] \le f(\tilde x) + \tfrac{\beta_k + \varepsilon_k}{2} \|\tilde x - \hat x_k \|^2 +\delta_k, 
\end{align*}
for any $\tilde x \in X$. In particular, setting $\tilde x = \gamma_k x + (1 - \gamma_k) \tilde x_{k-1}$ leads to
\begin{align}
& \EE[f(\tilde x_k) + \tfrac{\beta_k}{2} \|\tilde x_k - \hat x_k\|^2 +\tfrac{\alpha_k}{2} \|\gamma_k x + (1 - \gamma_k) \tilde x_{k-1} - x_t\|^2] \nonumber\\
\le & f(\gamma_k x + (1 - \gamma_k) \tilde x_{k-1}) + \tfrac{\beta_k + \varepsilon_k}{2} \|\gamma_k x + (1 - \gamma_k) \tilde x_{k-1} - \hat x_k \|^2 +\delta_k\nonumber\\
\overset{(a)}{\leq} &  \gamma_k f(x) + (1-\gamma_k) f(\tilde x_{k-1}) - \tfrac{\mu \gamma_k (1-\gamma_k)}{2} \|x - \tilde x_{k-1}\|^2 \nonumber\\
& \quad + \tfrac{\beta_k +\varepsilon_k}{2} \|\gamma_k x + (1 - \gamma_k) \tilde x_{k-1} - \hat x_k \|^2 +\delta_k, \nonumber
\end{align}
where inequality $(a)$ follows from \eqref{eq:phi_convexity}.
Simple rearrangements of the above relation yields 
\begin{align*}
& \EE \sbr{f(\tilde{x}_k) - f(x) ) + \frac{\alpha_k}{2}  \norm{\gamma_k x + (1-\gamma_k) \tilde{x}_{k-1} - x_k }^2
 }
+ \frac{\gamma_k(1-\gamma_k) \mu}{2} \norm{\tilde{x}_{k-1} -x}^2  \\
\leq & 
(1-\gamma_k) \sbr{f(\tilde{x}_{k-1}) - f(x)} 
+ \frac{\beta_k + \varepsilon_k}{2} \norm{\gamma_k x + (1-\gamma_k) \tilde{x}_{k-1} - \hat{x}_k}^2 + \delta_k.
\end{align*}
 Substituting the definition of $\hat{x}_k$ in \eqref{eq:define_hat_x}  into the above relation gives 
\begin{align}
& \EE \sbr{f(\tilde{x}_k) - f(x) )  + \frac{\alpha_k}{2}  \norm{\gamma_k x + (1-\gamma_k) \tilde{x}_{k-1} - x_k }^2
+ \frac{\gamma_k(1-\gamma_k) \mu}{2} \norm{\tilde{x}_{k-1} -x}^2 }  \nonumber \\
\leq & 
(1-\gamma_k) \sbr{f(\tilde{x}_{k-1}) - f(x)} 
+ \frac{(\beta_k + \varepsilon_k) \gamma_k^2 }{2} \norm{ x  - \overline{x}_{k-1}}^2 + \delta_k. \label{min_recursion_raw_bf_aggregation}
\end{align}
It remains to note that 
\begin{align}
&  \frac{\alpha_k}{2}  \norm{\gamma_k x + (1-\gamma_k) \tilde{x}_{k-1} - x_k }^2
+ \frac{\gamma_k(1-\gamma_k) \mu}{2} \norm{\tilde{x}_{k-1} -x}^2 \nonumber \\
= & \frac{\alpha_k \gamma_k^2}{2}  \norm{ x + \frac{1-\gamma_k}{\gamma_k} \tilde{x}_{k-1} - \frac{1}{\gamma_k} x_k }^2
+ \frac{\gamma_k(1-\gamma_k) \mu}{2} \norm{\tilde{x}_{k-1} -x}^2 \nonumber \\
\overset{(b)}{\geq} & 
\frac{\alpha_k \gamma_k^2 + \gamma_k(1-\gamma_k) \mu}{2} 
\norm{
x - \frac{ (  \mu - \alpha_k) (1-\gamma_k)  }{\alpha_k \gamma_k + (1-\gamma_k) \mu} \tilde{x}_{k-1} 
- \frac{\alpha_k}{\alpha_k \gamma_k + (1 - \gamma_k) \mu} x_k
}^2 \nonumber  \\
\overset{(c)}{=} & 
\frac{\alpha_k \gamma_k^2 + \gamma_k(1-\gamma_k) \mu}{2} 
\norm{
x - \overline{x}_k
}^2, \label{min_aggregate_x}
\end{align}
where $(b)$ follows from convexity of $\norm{\cdot}^2$, and $(b)$ follows from the definition of $\overline{x}_k$ in \eqref{eq:define_bar_x}.
The desired inequality \eqref{eq:main_recursion} follows immediately by combining \eqref{min_recursion_raw_bf_aggregation} and \eqref{min_aggregate_x}.
\end{proof}

With Lemma \ref{convex_each_step_characterization} in place, the next proposition establishes the global convergence of Algorithm \ref{alg:basic_cat} with generic  parameter specification.

\begin{proposition} \label{the:cat_main}
Define \begin{align} \label{eq:def_Gamma}
\Gamma_k := 
\begin{cases}
1, & k=1,\\
(1-\gamma_k) \Gamma_{k-1}, & k \ge 2.
\end{cases}
\end{align}
Suppose that $\alpha_k, \beta_k, \gamma_k$, and $\varepsilon_k$ in Algorithm~\ref{alg:basic_cat}
satisfies
\begin{align} \label{eq:cat_main_assumption}
\gamma_1 = 1, ~ 
\tfrac{(\beta_k + \varepsilon_k) \gamma_k^2}{\Gamma_k} \le \tfrac{[\alpha_{k-1} \gamma_{k-1}  +  \mu (1-\gamma_{k-1})] \gamma_{k-1}}{\Gamma_{k-1}}, k \ge 2.
\end{align}
Then for any $K \ge 1$, we have
\begin{align*}
 \EE[f(\tilde x_K) - f(x) + \tfrac{[\alpha_K \gamma_K  +  \mu (1-\gamma_K)] \gamma_K}{2} \| x - \bar x_{K}\|^2] \leq \Gamma_K \left[\tfrac{(\beta_1 + \epsilon_1) \gamma_1^2}{2} \|x - \bar x_{0} \|^2 + \tsum_{k=1}^K \tfrac{\delta_k}{\Gamma_k}\right], \forall x \in X. 
\end{align*}
\end{proposition}

\begin{proof}
Dividing both sides of \eqref{eq:main_recursion} by $\Gamma_k$, we have
\begin{align*}
&\EE[\tfrac{1}{\Gamma_k}[f(\tilde x_k) - f(x)]  + \tfrac{[\alpha_k \gamma_k  +  \mu (1-\gamma_k)] \gamma_k}{2 \Gamma_k} \| x - \bar x_{k}\|^2] \\
&\le \tfrac{1-\gamma_k}{\Gamma_k} [f(\tilde x_{k-1}) - f(x)]  + \tfrac{(\beta_k + \varepsilon_k) \gamma_k^2}{2 \Gamma_k} \|x - \bar x_{k-1} \|^2 + \tfrac{\delta_k}{\Gamma_k}, ~ \forall k \ge 1.
\end{align*}
Taking total expectation on both sides, the desired claim then follows from summing up these inequalities while utilizing the definition of $\Gamma_k$ in \eqref{eq:def_Gamma} and the assumption in \eqref{eq:cat_main_assumption}.
\end{proof}

With Theorem \ref{the:cat_main} in place, we first proceed to obtain the concrete convergence characterization of the catalyst scheme (Algorithm \ref{alg:basic_cat}) applied to smooth objectives ($\mu = 0$).

\begin{proposition}\label{prop_smooth_opt_generic}
Suppose $\mu = 0$.
Run  \texttt{Catalyst}$(\cA)$ (Algorithm~\ref{alg:basic_cat}) with 
\begin{align}
\label{gamma_beta_choice_catalyst}
\gamma_k = \frac{2}{k+1}, ~ \beta_k =  \frac{(k+1)L}{k}.
\end{align}
In addition, suppose $\cbr{\alpha_k}$ is chosen such that there exists $\cbr{(\varepsilon_k, \delta_k)}$ certifying \eqref{eq:cp_inexact} with 
\begin{align}
  \alpha_k =  \beta_k (1+\varepsilon), ~
  \varepsilon_k & = \beta_k \varepsilon,  ~ \varepsilon \leq 1, \label{err_condition_catalyst_1} \\
  ~\delta_k & \leq \delta \label{err_condition_catalyst_2}.
 \end{align}
Then we have 
\begin{align}\label{eq_opt_smooth_generic}
\EE[f(\tilde x_K) - f(x^*)] 
\le \tfrac{4L }{ K^2}  \|x^* - \bar x_{0} \|^2 + 2 K \delta . 
\end{align}
\end{proposition}

\begin{proof}
One can readily verify that \eqref{eq:cat_main_assumption} holds with the above choice of $(\alpha_k, \beta_k, \varepsilon_k, \gamma_k, \delta_k)$, and 
$
\Gamma_k = \frac{2}{k (k+1)}$  for any $k \geq 1$.
The rest of the claims then follow from a direct application of Theorem \ref{the:cat_main}.
\end{proof}

For smooth and strongly convex objectives ($\mu > 0$), we now describe a simple restarting technique applied to Algorithm \ref{alg:basic_cat}. 
The resulting method, presented in Algorithm \ref{alg_restart_opt_catalyst}, consists of multiple epochs.
Each epoch corresponds to running  Algorithm \ref{alg:basic_cat} starting  from the output of the previous epoch.

\begin{algorithm}[H]
\caption{\texttt{R-Catalyst}$(\cA)$: restarting catalyst for strongly convex problems}
\begin{algorithmic}
\STATE{ {\bf Input:} initial point $x_{(0)}$, to-be-catalyzed method $\cA$, total number of epochs $E$, epoch length $\cbr{K_e}$.}
\FOR{epoch $e = 1, 2, \ldots, E$}
\STATE{
Let $x_{(e)}$ be the output of  running $\texttt{Catalyst} (\cA)$ starting from $x_{(e-1)}$ for  $K_e$ iterations. 
}
\ENDFOR
\end{algorithmic} \label{alg_restart_opt_catalyst}
\end{algorithm}

We proceed to establish the convergence of \texttt{R-Catalyst}$(\cA)$ (Algorithm \ref{alg_restart_opt_catalyst}) when $\mu > 0$.

\begin{proposition}\label{prop_rcatalyst_general}
Suppose $\mu > 0$. Within the $e$-th epoch of \texttt{R-Catalyst}$(\cA)$ (Algorithm \ref{alg_restart_opt_catalyst}),  run
\texttt{Catalyst}$(\cA)$ (Algorithm~\ref{alg:basic_cat})  with 
\begin{align*}
 \gamma_{k} = \frac{2}{k+1}, ~ \beta_{k} =  \frac{(k+1)L}{k}
\end{align*}
for a total of 
$K_e = K \coloneqq 6 \sqrt{{L}/{\mu}}$ iterations.
In addition, suppose $\cbr{\alpha_k}$ is chosen such that there exists $\cbr{(\varepsilon_k, \delta_k)}$ certifying \eqref{eq:cp_inexact} with 
\begin{align}\label{err_condition_r_catalyst}
 \varepsilon_k = \beta_k  \varepsilon, ~ \alpha_k =  \beta_k (1+\varepsilon) ,  ~ \varepsilon \leq 1, 
\end{align}
and  
\begin{align}\label{err_condition_r_catalyst_eps_delta}
 \delta_{k} \leq \delta \coloneqq \frac{2^{-e-2}}{K} \sbr{f(x_{(0)}) - f(x^*) }.
 \end{align}
Then we have 
\begin{align*}
\EE \sbr{ f(x_{(e)}) - f(x^*) } \leq 2^{-e} \sbr{f(x_{(0)}) - f(x^*) }, ~ \forall e \geq 0.
\end{align*}
\end{proposition}

\begin{proof}
We proceed with an inductive argument. The claim clearly holds at $e = 0$. 
Now suppose the claim holds for $e-1$ for any $e \geq 1$, then 
\begin{align*}
\EE \sbr{ f(x_{(e)}) - f(x^*) }
& \overset{(a)}{\leq} \frac{4L}{ K^2} \EE \sbr{ \norm{x^* - x_{(e-1)}}^2 }  + 2 K \delta \\
& \overset{(b)}{\leq} \frac{8 L}{ \mu K^2} \EE \sbr{f(x_{(e-1)}) - f(x^*)} + 2 K \delta \\
& \overset{(c)}{\leq} \frac{1}{4} \EE \sbr{f(x_{(e-1)}) - f(x^*)}  + 2^{-e-1} \sbr{f(x_{(0)}) - f(x^*) } \\
& \overset{(d)}{\leq}  2^{-e} \sbr{f(x_{(0)}) - f(x^*) } ,
\end{align*} 
where $(a)$ follows from the direct application of  \eqref{eq_opt_smooth_generic} in Proposition \ref{prop_smooth_opt_generic},
$(b)$ follows from the strong convexity of $f$, 
$(c)$ follows from the definition of $K$ and $\delta$,
and $(d)$ follows from the induction hypothesis.
\end{proof}

\subsection{Catalyst for Stochastic Gradient Descent}\label{subsec_catalyst_sgd}

Up to now, the proposed catalyst scheme is a conceptual method, as it requires the to-be-catalyzed method $\cA$ to satisfy certain error condition \eqref{eq:cp_inexact} when solving the proximal step \eqref{eq:cp_subproblem}.
We now introduce a simple first-order method with this desired capability.
To proceed, we assume the access to a stochastic first-order oracle, defined as follows. 
\begin{definition}
The stochastic first-order oracle $\texttt{SFO}(f;x, \xi)$ for a differentiable function $f$ outputs the stochastic gradient $\nabla f(x; \xi)$ such that 
\begin{align*}
\EE_{\xi} \sbr{\nabla f(x; \xi)} = \nabla f(x), ~ \EE_{\xi} \norm{\nabla f(x; \xi) - \nabla f(x) }_*^2 \leq \sigma^2.
\end{align*}  
\end{definition}
Clearly,  $\texttt{SFO}(f;x, \xi)$ outputs the deterministic gradient almost surely when $\sigma = 0$.
Going forward, for any method, we refer to its total number of calls to \texttt{SFO} as its sample complexity. 
The proposed Procedure \ref{procedure_sgd}, termed \texttt{SGD}, can be viewed a simple adaptation of the stochastic gradient descent method for solving 
\begin{align}
\min_{u \in X}\{ \phi(u) \coloneqq f(u) + \tfrac{\beta}{2} \norm{u - \hat{x}}^2\}.
\end{align} 
Upon termination, Procedure \ref{procedure_sgd} returns both an ergodic mean of historical iterates and the last iterate. 
As will be clarified later, setting proper weights in the ergodic mean suffices to certify condition \eqref{eq:cp_inexact}. 
We use $L_\phi$ and $\mu_\phi$ denote the smoothness and strong-convexity modulus of $\phi$, respectively.

\begin{procedure}[H]
\caption{\texttt{SGD}$(\phi)$: stochastic gradient descent for $\min_{u \in X}\{ \phi(u) \coloneqq f(u) + \tfrac{\beta}{2} \norm{u - \hat{x}}^2\}$}
\begin{algorithmic}
\STATE {\bf Input:} stepsizes $\cbr{\eta_t}$, total number of steps $T > 0$, initial point $u_0 = \hat{x}$.
\FOR{t = 1, 2, \ldots, T}
\STATE Form  $g_{t-1} = \nabla f(u_{t-1}; \xi_{t-1}) + \beta (u_{t-1} - \hat{x})$, where $ \nabla f(u_{t-1}; \xi_{t-1}) = \texttt{SFO}(f; u_{t-1}, \xi_{t-1})$. 
\STATE $u_{t} = \argmin_{w \in X} \inner{g_{t-1}}{u} + \frac{1}{2\eta_t} \norm{u - u_{t-1}}^2.$
\ENDFOR
\STATE{Compute $ \overline{u}_T = \frac{\Lambda_T}{ (1-\Lambda_T)} \tsum_{t=1}^T \frac{ \eta_t \mu_\phi}{\Lambda_t} u_t$,
with $\Lambda_t$ defined in as \eqref{eq_def_Lambda}.}
\STATE {\bf Output:} $(\overline{u}_T, u_T)$.
\end{algorithmic}\label{procedure_sgd}
\end{procedure}

We next establish the generic convergence properties of Procedure~\ref{procedure_sgd}.

\begin{lemma}\label{prop_sgd_generic_convergence}
Let $\eta_t < 1/L_\phi$ and define 
\begin{align}\label{eq_def_Lambda}
\Lambda_t = \begin{cases}
(1-\eta_t \mu_\phi) \Lambda_{t-1}, ~ & t \geq 1, \\
1 ~ & t = 0.
\end{cases}
\end{align}
Then for any $T \geq 1$, we have 
 \begin{align*}
 \EE \sbr{\phi(\overline{u}_T) - \phi(u)}
 + \frac{\mu_\phi}{2(1-\Lambda_T)} \EE \norm{u - u_T}^2 
 \leq \frac{\Lambda_T \mu_\phi} {2(1-\Lambda_T)} \norm{u_0 - u}^2 + \frac{\Lambda_T } {2(1-\Lambda_T)} \tsum_{t=1}^T  \frac{\mu_\phi \eta_t^2 \sigma^2}{\Lambda_t (1 - L_\phi\eta_t)}.
 \end{align*}
\end{lemma}

The proof of Lemma \ref{prop_sgd_generic_convergence} is deferred to Appendix \ref{sec_supp}.
With Lemma \ref{prop_sgd_generic_convergence} in place, we can proceed to establish that Procedure \ref{procedure_sgd} produces solutions satisfying condition \eqref{eq:cp_inexact} with proper parameter specifications.

\begin{proposition}\label{prop_sgd_err_condition}
Let $(\overline{u}_T, u_T)$ be the output  of \texttt{SGD}$(\phi)$ (Procedure \ref{procedure_sgd}),
then  
\begin{align*}
\EE\sbr{\phi(\overline{u}_T) -\phi(u) + \tfrac{ \alpha \mu_\phi}{2}  \norm{\tilde u - u_T}^2} \leq  \tfrac{ \varepsilon \mu_\phi }{2}  \norm{u  - \hat{x}}^2 + \delta, ~ \forall u \in X, 
\end{align*}
with 
\begin{align}\label{sgd_err_condition_alpha_epsilon_delta}
\varepsilon = \frac{ \Lambda_T }{1 - \Lambda_T}, ~ 
\alpha = \frac{1}{1 - \Lambda_T} \equiv  1 + \varepsilon,~ 
\delta = \frac{\Lambda_T } {2(1-\Lambda_T)} \tsum_{t=1}^T  \frac{\mu_\phi \eta_t^2 \sigma^2}{\Lambda_t (1 - L_\phi\eta_t)}.
 \end{align}
 In particular, let
 \begin{align}\label{sgd_err_condition_param_choice}
 \eta_t = \frac{2}{\mu_\phi (t + t_0)} , ~  t_0 \geq \frac{4 L_\phi}{\mu_\phi}.
\end{align}
 Then for $T \geq t_0$, we have 
 \begin{align}\label{sgd_err_condition_epsilon_delta_ub}
 \varepsilon \leq  1, ~
 \delta
\leq \frac{32 \sigma^2}{\mu_\phi T}.
 \end{align}
 \end{proposition}

The proof of Proposition \ref{prop_sgd_err_condition} can be found in Appendix \ref{sec_supp}.
We are now ready to establish the SFO complexity of the proposed \texttt{Catalyst}$(\texttt{SGD})$ applied to \eqref{cp}  with smooth objectives.

\begin{theorem}\label{thrm_Tonsc_opt}
Suppose $\mu = 0$.
For any $\epsilon > 0$, 
run  \texttt{Catalyst}$(\texttt{SGD})$ with  
\begin{align*}
K = 4 \sqrt{\frac{L \norm{x^* - \overline{x}_0}^2}{\epsilon}}, ~ 
\gamma_k = \frac{2}{k+1}, ~ \beta_k =  \frac{(k+1)L}{k},
~ \alpha_k = \frac{\beta_k}{1- \Lambda_{T} },
\end{align*}
where 
\begin{align*}
T = 8 + \frac{32 \sigma^2 K}{L \epsilon}, ~ 
\Lambda_T = \frac{90}{(T+9) (T +10)}.
\end{align*} 
At the $k$-th iteration, let proximal step \eqref{eq:cp_subproblem} be solved by running \texttt{SGD}$(\phi_k)$ with stepsize 
\begin{align*}
 \eta_t = \frac{2}{\beta_k (t + 8)}.
\end{align*}
Then we have 
$
 \EE[f(\tilde x_K) - f(x^*)]  \leq \epsilon.
$
The number of calls to \texttt{SFO} is bounded by 
\begin{align*}
\cO \rbr{
 \sqrt{\frac{L \norm{x^* - \overline{x}_0}^2}{\epsilon}} + 
 \frac{\sigma^2 \norm{x^* - \overline{x}_0}^2}{\epsilon^2}
 }.
\end{align*}
\end{theorem}

\begin{proof}
 Given the choice of $\cbr{(\gamma_k, \beta_k)}$, 
 suppose \eqref{err_condition_catalyst_1} and \eqref{err_condition_catalyst_2}  are satisfied with $\delta = \frac{\epsilon}{4K}$,    we can invoke Proposition \ref{prop_smooth_opt_generic} and obtain
 \begin{align*}
 \EE[f(\tilde x_K) - f(x^*)] 
\le \tfrac{4L }{ K^2}  \norm{x^* - \bar x_{0} }^2 + 2 K \delta  \leq \epsilon,
 \end{align*}
 where the last inequality follows from the definition of $K$ and $\delta$.
 It remains to determine the number of steps $T$ and stepsizes $\cbr{\eta_t}$ required by \texttt{SGD} at iteration $k$ so that \eqref{err_condition_catalyst_1} and  \eqref{err_condition_catalyst_2} hold with the choice of $\delta$, and consequently determine the value of $\alpha_k$ therein. 

Given the choice of $\beta_k$,  the subproblem \eqref{eq:cp_subproblem} is $\beta_k$ strongly convex and $2 \beta_k$ smooth. 
Combining this observation with \eqref{sgd_err_condition_alpha_epsilon_delta}, \eqref{sgd_err_condition_param_choice}  and \eqref{sgd_err_condition_epsilon_delta_ub}, 
together with the choice of $\beta_k = \frac{(k+1)L}{k}$, 
it can be readily verified that \eqref{err_condition_catalyst_1} is satisfied by choosing 
\begin{align*}
\eta_t = \frac{2}{ \beta_k (t + 8)}, ~ 
 T \geq   8,
 ~  \alpha_k = \frac{\beta_k}{1 - \Lambda_T }.
\end{align*}
In addition, to satisfy \eqref{err_condition_catalyst_2}, from \eqref{sgd_err_condition_epsilon_delta_ub} and the choice of $\beta_k$ it suffices to choose 
\begin{align*}
T = 8 + \frac{32 \sigma^2}{L \delta} .
\end{align*} 
 The total number of calls to SFO is bounded by 
 \begin{align*}
 T \cdot K =
 \cO \rbr{
 K + \frac{\sigma^2}{L \delta} K
 }
 =
 \cO \left(
 \sqrt{\frac{L \norm{x^* - \overline{x}_0}^2}{\epsilon}} + 
 \frac{\sigma^2 \norm{x^* - \overline{x}_0}^2}{\epsilon^2}
 \right).
 \end{align*}
 The proof is then completed.
\end{proof}

As the last result of this section, we  proceed to establish the SFO complexity of \texttt{R-Catalyst}$(\texttt{SGD})$ applied to \eqref{cp} with a smooth and strongly convex objective.

\begin{theorem}\label{thrm_sc_opt}
Suppose $\mu > 0$.
For any $\epsilon > 0$, 
run  \texttt{R-Catalyst}$(\texttt{SGD})$ for a total of  $E = \log_{2}(\frac{f(x_{(0)}) - f(x^*)}{\epsilon})$ epochs.
The $e$-th epoch corresponds to running \texttt{Catalyst}$(\texttt{SGD})$, initialized at $x_{(e)}$, with parameters
\begin{align*}
K = 6 \sqrt{\frac{L}{\mu}}, ~ \gamma_k = \frac{2}{k+1}, ~ \beta_k =  \frac{(k+1)L}{k},
~ ~\alpha_k = \frac{\beta_k }{1 - \Lambda_{T} },
\end{align*}
where 
\begin{align*}
T = 8 + \frac{128 \sigma^2 2^{e} K}{L (f(x_{(0)} )- f(x^*))}, ~
\Lambda_T = \frac{90}{(T+9) (T +10)}.
\end{align*}
In addition, at  the $k$-th iteration  of the $e$-th epoch, the proximal step \eqref{eq:cp_subproblem} is solved by running \texttt{SGD}$(\phi_k)$ with 
$
 \eta_t = \frac{2}{\beta_k (t + 8)}.
$
Then we have 
\begin{align*}
 \EE[f(\tilde x_K) - f(x^*)]  \leq \epsilon.
\end{align*}
In addition, the number of calls to \texttt{SFO} is bounded by 
\begin{align}\label{r_catalyst_sgd_strongly_convex_complexity}
\cO \rbr{
 \sqrt{\frac{L}{\mu}} \log_{2} \left(\frac{f(x_{(0)}) - f(x^*)}{\epsilon}\right)
+ \frac{\sigma^2}{\mu \epsilon}
 }.
\end{align}

\end{theorem}

\begin{proof}
Given the choice of $\cbr{(\gamma_k, \beta_k)}$, 
suppose \eqref{err_condition_r_catalyst} and \eqref{err_condition_r_catalyst_eps_delta} are satisfied, 
we can invoke Proposition \ref{prop_rcatalyst_general} and obtain
$
\EE \sbr{ f(x_{(e)}) - f(x^*)  } \leq \epsilon
$,
given the choice of $E$.
 It remains to determine the number of steps $T$ and stepsizes $\cbr{\eta_t}$ required by \texttt{SGD} at iteration $k$ of epoch $e$ so that \eqref{err_condition_r_catalyst} and \eqref{err_condition_r_catalyst_eps_delta} are satisfied, and consequently determine the value of $\alpha_k$ therein. 
Given the choice of $\beta_k$,  the subproblem \eqref{eq:cp_subproblem} is $\beta_k$ strongly convex and $2 \beta_k$ smooth. 
 Hence in view of  \eqref{sgd_err_condition_alpha_epsilon_delta}, \eqref{sgd_err_condition_param_choice}, \eqref{sgd_err_condition_epsilon_delta_ub} in Proposition \ref{prop_sgd_err_condition}, 
 to satisfy \eqref{err_condition_r_catalyst} 
 it suffices to choose  
$
  \eta_t = \frac{2}{ \beta_k  (t + 8)}$,  
$ T \geq   8$, 
 $
  \alpha_k = \frac{\beta_k }{1 - \Lambda_T}.
$
 In addition, from \eqref{sgd_err_condition_epsilon_delta_ub}, it can be readily seen that to satisfy \eqref{err_condition_r_catalyst_eps_delta} it suffices to take 
 \begin{align*}
  \frac{32 \sigma^2}{ \beta_k T} \leq \frac{2^{-e-2}}{K} \sbr{f(x_{(0)}) - f(x^*) }.
 \end{align*}
Given the choice of $\beta_k$ the above condition holds by taking 
$
T = 8 + \frac{128 \sigma^2 2^{e} K}{L (f(x_{(0)} )- f(x^*))}.
$
Consequently, the total number of calls to SFO is bounded by 
\begin{align*}
\tsum_{e=1}^E K \cdot T & = \tsum_{e=1}^E 8 K  + \frac{128 \sigma^2 2^{e} K^2}{L (f(x_{(0)} )- f(x^*))}  = \cO \left(
\sqrt{\frac{L}{\mu}} \log_{2} \left(\frac{f(x_{(0)}) - f(x^*)}{\epsilon}\right)
+ \frac{\sigma^2}{\mu \epsilon}
\right).
\end{align*}
The proof is then completed.
\end{proof}

As stated in Theorem \ref{thrm_sc_opt}, the concrete parameter choice of \texttt{R-Catalyst}(\texttt{SGD}) requires the knowledge of the initial optimality gap.
When this quantity is unknown,  one can instead use its overestimate. 
In view of \eqref{r_catalyst_sgd_strongly_convex_complexity}, the price for potential overestimation is only a logarithmic factor increase in  the overall complexity of \texttt{R-Catalyst}(\texttt{SGD}).

A few remarks are in order before we conclude our discussions in this section. 
In view of Theorem \ref{thrm_Tonsc_opt} and \ref{thrm_sc_opt}, 
for the deterministic setting  where $\sigma = 0$, the proposed catalyst scheme attains the optimal $\cO(\sqrt{L/\epsilon})$ (resp. $\cO(\sqrt{L/\mu}\log(1/\epsilon))$) iteration complexity for solving \eqref{cp} with a smooth (resp. smooth and strongly-convex) objective.
For the stochastic setting, 
the proposed scheme attains the optimal  sample complexity   $\cO(\sigma^2 /\epsilon^2)$   (resp. $\cO(\sigma^2/(\mu \epsilon))$) for smooth  (resp. smooth and strongly-convex) objectives.
Consequently, the catalyst scheme presented in this section is able to accelerate simple (stochastic) gradient method and simultaneously obtain optimal complexities in both deterministic and stochastic settings \cite{lan2012optimal}.


\section{Extragradient for Minimax Optimization}\label{sec_eg}

In this section we turn our attention back to the minimax problem \eqref{def_problem}. 
For the purpose of our discussion, we  assume within this section that the minimax problem is strongly-convex-strongly-concave. 
That is, for some $\mu_p, \mu_d > 0$, 
\begin{align}
F(x_1, y) - F(x_2, y) - \langle \nabla_x F(x_2, y), x_1-x_2 \rangle \ge \tfrac{\mu_p}{2} \|x_1 - x_2\|^2, \ \forall x_1, x_2 \in X, \label{sc_x_seg}\\ 
F(x, y_1) - F(x, y_2) - \langle \nabla_y F(x, y_2), y_1-y_2 \rangle \le -\tfrac{\mu_d}{2} \|y_1 - y_2\|^2, \ \forall y_1, y_2 \in Y. \label{sc_y_seg}
\end{align}
Let us define $\mu \coloneqq \min \cbr{\mu_p, \mu_d} > 0$.
We next introduce a variant of extragradient method for problem \eqref{def_problem}, which we term as regularized extragradient ({REG}) in the ensuing discussion.
In particular, we will establish in this section that REG obtains optimal complexities in both deterministic and stochastic settings for solving smooth and strongly-monotone variational inequalities. 
More importantly, it turns out that REG and its stochastic variant come with a convergence characterization
that simultaneously controls the gap value and the distance to any reference point chosen a posteriori, which
naturally aligns with the error condition required by the catalyst scheme for problem \eqref{def_problem}.
This property subsequently makes them ideal candidates as the to-be-catalyzed methods for the catalyst scheme, which we detail in  Section \ref{sec_minimax_deterministic} and \ref{sec_minimax_stoch}.
Notably, the aforementioned convergence characterization has not been established or reported before for other methods applied to smooth and strongly-monotone VIs (e.g. \cite{kotsalis2022simple, mokhtari2020unified, beznosikov2021distributed}),  which solely focus on the distance to the solution (saddle point).

\subsection{Deterministic Regularized Extragradient}

The deterministic regularized extragradient ({REG}) method is presented in Algorithm \ref{alg_meg_deterministic}. 
Similar to the SGD variant discussed in Section \ref{subsec_catalyst_sgd},
REG returns the the latest iterate and an ergodic average of historical iterates. 
With a slight overload of notations, for the remainder of the section we will  denote the weights in the construction of the ergodic average as $\cbr{\Lambda_t}$, defined by 
\begin{align}\label{def_lambda_minimax}
\Lambda_t = \begin{cases}
1, &~ t= 0; \\
\Lambda_{t-1} (1 + \mu \eta_{t-1}) , &~ t \geq 1.
\end{cases}
\end{align}
Clearly, REG differs from the original extragradient method \cite{korpelevich1976extragradient} by the  additional strongly-convex term $\norm{z - \hat{z}_t}^2$ in the ``extragradient'' step \eqref{eg_update_2}. 
As will be clarified later, the introduction of this element substantially improves the convergence of the extragradient method to linear rate for  strongly-convex-strongly-concave problems.

\begin{algorithm}[H]
\caption{\texttt{REG}$(F)$: regularized extragradient for $\min_{x \in X} \max_{y \in Y} F(x,y) $}
\begin{algorithmic}\label{alg_meg_deterministic}
\STATE {\bf Input:} stepsizes $\cbr{\eta_t}$, total number of steps $T > 0$, initial point $z_0 \in Z$.
\FOR{t = 0, 1, \ldots, T-1}
\STATE Define $G(z) = [\nabla_{x} F(z); -\nabla_{y} F(z)]$ for any $z \in Z$, and consequently perform update
\begin{align}
\hat{z}_t &= \argmin_{z \in Z} \eta_t \inner{G(z_t)}{z} + \frac{1}{2} \norm{z - z_t}^2; \label{eg_update_1} \\
z_{t+1} & = \argmin_{z \in Z}\eta_t \sbr{ \inner{G(\hat{z}_t)}{z} + \frac{\mu}{2} \norm{z - \hat{z}_t}^2} + \frac{1}{2} \norm{z - z_t}^2. \label{eg_update_2}
\end{align}
\ENDFOR
\STATE Construct $\overline{z}_T =  {\tsum_{t=0}^{T-1} \eta_t \Lambda_t \hat{z}_t }/({\tsum_{t=0}^{T-1} \eta_t \Lambda_t})$, with $\cbr{\Lambda_t}$ defined in \eqref{def_lambda_minimax}.
\STATE {\bf Output:}  $(\overline{z}_T, z_T)$.
\end{algorithmic}
\end{algorithm}

We proceed to establish the generic convergence properties of the {REG} method. 

\begin{lemma}\label{lemma_eg_deterministic}
Take stepsize $\eta_t \leq 1/L$.
Then for any $T \geq 1$, we have
\begin{align}\label{ineq_eg_deterministic_general}
F(\overline{x}_T, y) - F(x, \overline{y}_T) 
+ \frac{\mu \Lambda_T}{2 (\Lambda_T - \Lambda_0)} \norm{z - z_T}^2 
\leq \frac{\mu \Lambda_0}{2 (\Lambda_T - \Lambda_0)} \norm{z - z_0}^2.
\end{align}
\end{lemma}

\begin{proof}
First, the optimality condition of \eqref{eg_update_1} yields 
\begin{align}\label{eg_update_1_opt_condition}
\eta_t \inner{G(z_t)}{\hat{z}_t - z}
+ \frac{1}{2} \norm{\hat{z}_t - z_t}^2 + \frac{1}{2} \norm{\hat{z}_t - z}^2 
\leq \frac{1}{2} \norm{z - z_t}^2.
\end{align}
Similarly, the optimality condition of \eqref{eg_update_2} yields
\begin{align}\label{eg_update_2_opt_condition}
& \eta_t \inner{G(\hat{z}_t)}{z_{t+1} - z} + \frac{\mu \eta_t + 1}{2} \norm{z - z_{t+1}}^2 
+ \frac{\mu \eta_t}{2} \norm{z_{t+1} - \hat{z}_t}^2 + \frac{1}{2} \norm{z_{t+1} - z_t}^2  \nonumber \\
\leq & \frac{\mu \eta_t}{2} \norm{z - \hat{z}_t}^2 + \frac{1}{2} \norm{z - z_t}^2. 
\end{align}
By taking $z = z_{t+1}$ in \eqref{eg_update_1_opt_condition}, and combining with \eqref{eg_update_2_opt_condition}, we obtain 
\begin{align}
& \eta_t \inner{G(z_t)}{\hat{z}_t - z_{t+1}}
+ \frac{1}{2} \norm{\hat{z}_t - z_t}^2 + \frac{1}{2} \norm{\hat{z}_t - z_{t+1}}^2
+ \eta_t \inner{G(\hat{z}_t)}{z_{t+1} - z} + \frac{\mu \eta_t + 1}{2} \norm{z - z_{t+1}}^2 
 \nonumber  \\
\leq &  \frac{\mu \eta_t}{2} \norm{z - \hat{z}_t}^2 + \frac{1}{2} \norm{z - z_t}^2.  \label{eg_recursion_raw}
\end{align}
In addition, one also has 
\begin{align}
& \eta_t \sbr{ \inner{G(z_t)}{\hat{z}_t - z_{t+1}} +  \inner{G(\hat{z}_t)}{z_{t+1} - z}  } + \frac{1}{2} \norm{\hat{z}_t - z_t}^2 + \frac{1}{2} \norm{\hat{z}_t - z_{t+1}}^2  \nonumber \\
= & \eta_t  \sbr{ \inner{G(\hat{z}_t)}{\hat{z}_t - z} + \inner{G(z_t) - G(\hat{z}_t)}{\hat{z}_t - z_{t+1}} }  + \frac{1}{2} \norm{\hat{z}_t - z_t}^2 + \frac{1}{2} \norm{\hat{z}_t - z_{t+1}}^2  \nonumber \\
\overset{(a)}{\geq} & \eta_t \sbr{ F(\hat{x}_t, y) - F(x, \hat{y}_t) + \frac{\mu}{2} \norm{\hat{z}_t - z}^2 
- L\norm{z_t - \hat{z}_t} \norm{\hat{z}_t - z_{t+1}} } + \frac{1}{2} \norm{\hat{z}_t - z_t}^2 + \frac{1}{2} \norm{\hat{z}_t - z_{t+1}}^2 \nonumber \\
\overset{(b)}{\geq} & \eta_t  \sbr{ F(\hat{x}_t, y) - F(x, \hat{y}_t) + \frac{\mu}{2} \norm{\hat{z}_t - z}^2 }, \label{eg_extra_step_approx}
\end{align}
where $(a)$ follows from $G$ being $L$-Lipschitz, and $F$ being
 $\mu$-strongly-convex w.r.t. $x$ and strongly-concave w.r.t. $y$;  $(b)$ follows from H\"{o}lder's inequality together with $L\eta_t \leq 1$.
By combining \eqref{eg_recursion_raw} and \eqref{eg_extra_step_approx}, we obtain 
\begin{align*}
\eta_t \sbr{ F(\hat{x}_t, y) - F(x, \hat{y}_t) }
+ \frac{\mu \eta_t + 1}{2} \norm{z - z_{t+1}}^2 
\leq   \frac{1}{2} \norm{z - z_t}^2.
\end{align*}
Multiplying both sides by $\Lambda_t$, and making using of its definition, we obtain 
\begin{align*}
\eta_t \Lambda_t  \sbr{ F(\hat{x}_t, y) - F(x, \hat{y}_t) }
+ \frac{\Lambda_{t+1}}{2 } \norm{z - z_{t+1}}^2 \leq \frac{\Lambda_t}{2 } \norm{z - z_t}^2.
\end{align*}
Taking the telescopic sum of the above inequality from $t=0$ to $T-1$ yields 
\begin{align*}
\tsum_{t=0}^{T-1} \eta_t \Lambda_t \sbr{ F(\hat{x}_t, y) - F(x, \hat{y}_t) }
+ \frac{\Lambda_T}{2} \norm{z - z_T}^2 
\leq \frac{\Lambda_0}{2} \norm{z - z_0}^2.
\end{align*}
Further dividing both sides of the above relation by $\tsum_{t=0}^{T-1} \eta_t \Lambda_t$, while noting that $ \tsum_{t=0}^{T-1} \eta_t \Lambda_t = \rbr{\Lambda_T - \Lambda_0}/{\mu}$  from \eqref{def_lambda_minimax} 
and making use of the definition of $\overline{z}_T$, we obtain the desired claim.
\end{proof}

We are now ready to establish the linear convergence of REG with concrete parameter specifications.

\begin{theorem}
Suppose $\mu > 0$. Let $z^*$ be the unique saddle point of \eqref{def_problem}.
Choose $\eta_t = 1/L$ in the REG method, then 
\begin{align*}
\norm{z_T - z^*}^2 \leq \rbr{1 + \frac{\mu}{L}}^{-T} \norm{z_0 - z^*}^2.
\end{align*}
\end{theorem}
\begin{proof}
The desired claim is an immediate consequence of Lemma \ref{lemma_eg_deterministic}.
\end{proof}

It is worth mentioning that the above convergence of REG generalizes trivially to solving variational inequality 
\begin{align*}
\inner{G(z)}{z - z^*} \geq 0, ~ \forall z \in Z,
\end{align*}
when operator $G$ is smooth and strongly-monotone. 
Notably, the obtained iteration complexity by REG is optimal for this problem class.

\subsection{Stochastic Regularized Extragradient}

The stochastic variant of the REG method, which we term SREG (Algorithm \ref{alg_smeg}), replaces the deterministic gradient operator by its stochastic approximation within its update. 

\begin{algorithm}[H]
\caption{\texttt{SREG}$(F)$: stochastic extragradient for $\min_{x \in X} \max_{y \in Y} F(x,y) $}
\begin{algorithmic}\label{alg_smeg}
\STATE {\bf Input:} stepsizes $\cbr{\eta_t}$, total number of steps $T > 0$, initial point $z_0 \in Z$.
\FOR{t = 0, 1, \ldots, T-1}
\STATE Define ${G}(z, \xi) = [\nabla_{x} F(z; \xi); -\nabla_{y} F(z, \xi)]$. Sample $\xi_t, \hat{\xi}_t$, and update
\begin{align}
\hat{z}_t &= \argmin_{z \in Z} \eta_t \inner{G(z_t, \xi_t)}{z} + \frac{1}{2} \norm{z - z_t}^2; \label{eg_update_1_stoch_general} \\
z_{t+1} & = \argmin_{z \in Z}\eta_t \sbr{ \inner{G(\hat{z}_t, \hat{\xi}_t)}{z} + \frac{\mu}{2} \norm{z - \hat{z}_t}^2 } + \frac{1}{2} \norm{z - z_t}^2. \label{eg_update_2_stoch_general}
\end{align}
\ENDFOR
\STATE Construct $\overline{z}_T =  {\tsum_{t=0}^{T-1} \eta_t \Lambda_t \hat{z}_t }/({\tsum_{t=0}^{T-1} \eta_t \Lambda_t})$, with $\cbr{\Lambda_t}$ defined as in \eqref{def_lambda_minimax}.
\STATE {\bf Output:}  $(\overline{z}_T, z_T)$.
\end{algorithmic}\label{procedure_sgd_stoch}
\end{algorithm}

For notational simplicity going forward let us denote 
$\zeta_t = G(z_t, \xi_t) - G(z_t)$ and $ \hat{\zeta}_t = G(\hat{z}_t, \hat{\xi}_t) - G(\hat{z}_t)$.
The next lemma provides a basic characterization on each step of the SREG method.

\begin{lemma}\label{lemma_seg_one_step_recursion_general}
Suppose 
\begin{align*}
L \eta_t \leq 1/2, ~ t \geq 0.
\end{align*}
Then for any $t \geq 0$,  we have
\begin{align}
  \eta_t  \sbr{ F(\hat{x}_t, y) - F(x, \hat{y}_t)}   + \frac{\mu \eta_t + 1}{2} \norm{z - z_{t+1}}^2   + \eta_t \inner{\hat{\zeta}_t}{\hat{z}_t - z}   \leq \frac{1}{2} \norm{z - z_t}^2   + 4   \eta_t^2 \rbr{ \norm{\zeta_t}^2 + \norm{\hat{\zeta}_t}^2}.
\label{seg_one_step_recursion_general}
\end{align}
\end{lemma}

\begin{proof}
First, the optimality condition of \eqref{eg_update_1_stoch_general} yields 
\begin{align}\label{eg_update_1_opt_condition_stoch_general}
\eta_t \inner{G(z_t, \xi_t)}{\hat{z}_t - z}
+ \frac{1}{2} \norm{\hat{z}_t - z_t}^2 + \frac{1}{2} \norm{\hat{z}_t - z}^2 
\leq \frac{1}{2} \norm{z - z_t}^2.
\end{align}
Similarly, the optimality condition of \eqref{eg_update_2_stoch_general} yields
\begin{align}\label{eg_update_2_opt_condition_stoch_general}
& \eta_t \inner{G(\hat{z}_t, \hat{\xi}_t)}{z_{t+1} - z} + \frac{\mu \eta_t + 1}{2} \norm{z - z_{t+1}}^2  
+ \frac{1}{2} \norm{z_{t+1} - z_t}^2  \leq 
 \frac{\mu\eta_t}{2} \norm{z - \hat{z}_t}^2  + \frac{1}{2} \norm{z - z_t}^2. 
\end{align}
By taking $z = z_{t+1}$ in \eqref{eg_update_1_opt_condition_stoch_general}, and combining with \eqref{eg_update_2_opt_condition_stoch_general}, we obtain 
\begin{align}
& \eta_t \inner{G(z_t)}{\hat{z}_t - z_{t+1}}
+ \frac{1}{2} \norm{\hat{z}_t - z_t}^2 + \frac{1}{2} \norm{\hat{z}_t - z_{t+1}}^2
+ \eta_t \inner{G(\hat{z}_t)}{z_{t+1} - z} 
 + \frac{\mu \eta_t + 1}{2} \norm{z - z_{t+1}}^2  \nonumber \\
& ~~~  + \eta_t \inner{G(z_t, \xi_t) - G(z_t) }{\hat{z}_t - z_{t+1}} 
 +  \eta_t \inner{G(\hat{z}_t, \hat{\xi}_t) - G(\hat{z}_t)}{z_{t+1} - \hat{z}_t}
 + \eta_t \inner{G(\hat{z}_t, \hat{\xi}_t) - G(\hat{z}_t)}{\hat{z}_t - z}
  \nonumber  \\
\leq &  \frac{\mu \eta_t}{2} \norm{z - \hat{z}_t}^2  + \frac{1}{2} \norm{z - z_t}^2.  \nonumber
\end{align}
Then further applying Young's inequality to the above relation yields 
\begin{align}
& \eta_t \inner{G(z_t)}{\hat{z}_t - z_{t+1}}
+ \frac{1}{2} \norm{\hat{z}_t - z_t}^2 + \frac{1}{4} \norm{\hat{z}_t - z_{t+1}}^2
+ \eta_t \inner{G(\hat{z}_t)}{z_{t+1} - z} \nonumber  \\
& ~~~  + \frac{\mu \eta_t + 1}{2} \norm{z - z_{t+1}}^2   
 + \eta_t \inner{\hat{\zeta}_t}{\hat{z}_t - z}
  \nonumber  \\
\leq &  \frac{\mu \eta_t}{2} \norm{z - \hat{z}_t}^2  + \frac{1}{2} \norm{z - z_t}^2 
+ 2 \eta_t^2 \rbr{ \norm{\zeta_t}^2 + \norm{\hat{\zeta}_t}^2}.
 \label{eg_recursion_raw_stoch_general}
\end{align}
Following similar lines as in the proof of Lemma \ref{lemma_eg_deterministic}, one also has 
\begin{align}
& \eta_t \sbr{ \inner{G(z_t)}{\hat{z}_t - z_{t+1}} +  \inner{G(\hat{z}_t)}{z_{t+1} - z}  } + \frac{1}{2} \norm{\hat{z}_t - z_t}^2 + \frac{1}{4} \norm{\hat{z}_t - z_{t+1}}^2  \nonumber \\
= & \eta_t  \sbr{ \inner{G(\hat{z}_t)}{\hat{z}_t - z} + \inner{G(z_t) - G(\hat{z}_t)}{\hat{z}_t - z_{t+1}} }  + \frac{1}{2} \norm{\hat{z}_t - z_t}^2 + \frac{1}{4} \norm{\hat{z}_t - z_{t+1}}^2  \nonumber \\
\overset{(a)}{\geq} & \eta_t \sbr{ F(\hat{x}_t, y) - F(x, \hat{y}_t) + \frac{\mu}{2} \norm{\hat{z}_t - z}^2  
- L \norm{z_t - \hat{z}_t} \norm{\hat{z}_t - z_{t+1}} } + \frac{1}{2} \norm{\hat{z}_t - z_t}^2 + \frac{1}{4} \norm{\hat{z}_t - z_{t+1}}^2 \nonumber \\
\overset{(b)}{\geq} & \eta_t  \sbr{ F(\hat{x}_t, y) - F(x, \hat{y}_t) + \frac{\mu}{2} \norm{\hat{z}_t - z}^2  }, \label{eg_extra_step_approx_stoch_general}
\end{align}
where $(a)$ follows from $G$ being $L$-Lipschitz, and $F$ is  strongly-convex-strongly-concave with modulus $\mu$;  $(b)$ follows from H\"{o}lder's inequality together with $L\eta_t \leq {1}/{2}$.
The desired claim follows by combining \eqref{eg_recursion_raw_stoch_general} and \eqref{eg_extra_step_approx_stoch_general}.
\end{proof}

With Lemma \ref{lemma_seg_one_step_recursion_general} in place, we are ready to establish the generic convergence properties of SREG. 

\begin{lemma}\label{lemma_seg_general_convergence}
Let $\cbr{\Lambda_t}$ be defined as in \eqref{def_lambda_minimax}.
Choose in the SREG method that 
\begin{align}\label{seg_stepsize_condition}
L \eta_t \leq 1/2, ~ t \geq 0.
\end{align}
Then for any $T \geq 1$, we have 
\begin{align*}
& \EE \sbr{ F(\overline{x}_T, y ) - F(x, \overline{y}_T)
+ \frac{\mu \Lambda_T}{2(\Lambda_T - \Lambda_0)} \norm{z - z_T}^2  } \\
\leq & 
\frac{\mu \Lambda_0 }{2(\Lambda_T - \Lambda_0)} \norm{z - z_0}^2 
+ \frac{8 \mu \sigma^2}{\Lambda_T - \Lambda_0}  \tsum_{t=0}^{T-1}  \eta_t^2 \Lambda_t .
\end{align*}
\end{lemma}

\begin{proof}
Multiplying both sides of \eqref{seg_one_step_recursion_general} by $\Lambda_t$ yields 
\begin{align*}
 \eta_t \Lambda_t  \sbr{ F(\hat{x}_t, y) - F(x, \hat{y}_t)}   + \frac{(\mu \eta_t + 1)\Lambda_t}{2} \norm{z - z_{t+1}}^2   + \eta_t \Lambda_t \inner{\hat{\zeta}_t}{\hat{z}_t - z}   \leq \frac{\Lambda_t}{2} \norm{z - z_t}^2   + 4   \eta_t^2 \Lambda_t \rbr{ \norm{\zeta_t}^2 + \norm{\hat{\zeta}_t}^2}.
\end{align*}
Applying the definition of $\cbr{\Lambda_t}$, we can now take the telescopic sum of the above relation and obtain 
\begin{align*}
& \tsum_{t=0}^{T-1} \eta_t \Lambda_t  \sbr{ F(\hat{x}_t, y) - F(x, \hat{y}_t)}
+ \frac{(\mu \eta_{T-1} + 1) \Lambda_{T-1}}{2} \norm{z - z_T}^2 
+ \tsum_{t=0}^{T-1} \eta_t \Lambda_t  \inner{\hat{\zeta}_t}{\hat{z}_t - z}  \\
\leq & 
\frac{\Lambda_0}{2} \norm{z - z_0}^2 
+ 4 \tsum_{t=0}^{T-1}  \eta_t^2 \Lambda_t \rbr{ \norm{\zeta_t}^2 + \norm{\hat{\zeta}_t}^2} .
\end{align*}
Dividing both sides of the above relation by $\tsum_{t=0}^{T-1} \eta_t \Lambda_t = \frac{1}{\mu} \rbr{\Lambda_T - \Lambda_0}$, 
and making use of $F$ being convex-concave, 
we obtain 
\begin{align*}
& F(\overline{x}_T, y ) - F(x, \overline{y}_T)
+ \frac{\mu \Lambda_T}{2(\Lambda_T - \Lambda_0)} \norm{z - z_T}^2  \\
\leq & 
\frac{\mu \Lambda_0 }{2(\Lambda_T - \Lambda_0)} \norm{z - z_0}^2 
+ \frac{4 \mu}{\Lambda_T - \Lambda_0}  \tsum_{t=0}^{T-1}  \eta_t^2 \Lambda_t \rbr{ \norm{\zeta_t}^2 + \norm{\hat{\zeta}_t}^2} 
- \frac{\mu}{\Lambda_T -\Lambda_0} \tsum_{t=0}^{T-1} \eta_t \Lambda_t  \inner{\hat{\zeta}_t}{\hat{z}_t - z}.
\end{align*}
The desired claim follows after taking expectation on both sides of the above relation, 
and noting that $\EE [\inner{\hat{\zeta}_t}{\hat{z}_t - z}] = 0$. 
\end{proof}

 We now specify detailed choice of $\cbr{\eta_t}$ and obtain the concrete convergence characterization of SREG as an immediate consequence of Lemma \ref{lemma_seg_general_convergence}.

\begin{theorem}\label{thrm_seg_sublinear_deterministic_err}
Define $t_0 = 4 \ceil{{L}/{\mu}}$, and choose 
$
\eta_t = \frac{2}{\mu (t + t_0 + 1)}.
$
Then 
\begin{align*}
& \EE \sbr{
 \norm{z^* - z_T}^2  } \leq 
\frac{6  t_0^2}{T^2} \norm{z^* - z_0}^2 
+ \frac{768 \sigma^2}{\mu^2 T}.
\end{align*}
\end{theorem}

\begin{proof}
Clearly, \eqref{seg_stepsize_condition} is satisfied by the choice of $\cbr{\eta_t}$.
Direct computation also yields 
$
\Lambda_t = \frac{(t+t_0 + 1)(t + t_0 + 2)}{(t_0 + 1) (t_0 + 2)}.
$
Applying Lemma \ref{lemma_seg_general_convergence} with $(x,y) = (x^*, y^*)$ then yields  
\begin{align*}
& \EE \sbr{ F(\overline{x}_T, y^* ) - F(x^*, \overline{y}_T)
+ \frac{\mu}{2} \norm{z^* - z_T}^2  } \leq 
\frac{\mu \Lambda_0 }{2(\Lambda_T - \Lambda_0)} \norm{z^* - z_0}^2 
+ \frac{8 \mu \sigma^2}{\Lambda_T - \Lambda_0}  \tsum_{t=0}^{T-1}  \eta_t^2 \Lambda_t .
\end{align*}
The desired claim then follows from the above relation and the following simple observation:
\begin{align*}
\Lambda_T - \Lambda_0 =  \frac{(T+t_0 + 1)(T + t_0 + 2)}{(t_0 + 1) (t_0 + 2)} - 1 \geq \frac{T^2}{6 t_0^2} ,
~ 
\tsum_{t = 0}^{T-1} \eta_t^2 \Lambda_t 
= \frac{4 (t+ t_0 +2)}{\mu^2 (t_0 + 1) (t_0 + 2) (t+t_0 + 1)} \leq \frac{8 T}{\mu^2 t_0^2}.
\end{align*}
\end{proof}

Theorem \ref{thrm_seg_sublinear_deterministic_err} implies an $\cO(\frac{L \norm{z^* - z_0}}{ \mu \sqrt{ \epsilon}} + \frac{\sigma^2}{\mu^2 \epsilon})$ iteration and sample complexities  for SREG to obtain an $\epsilon$-optimal solution. 
We next proceed to show that with a proper restarting technique, one can indeed substantially improve the reduction of the deterministic error within the SREG method.

\begin{theorem}\label{theorem_stoch_smeg}
Consider the multi-epoch version of SREG defined as follows. 
For any $e > 0$, within the $e$-th epoch, run SREG starting from $z_{(e-1)}$   for $T_e$ steps with parameters 
\begin{align*}
T_e =  \frac{6  t_0^2}{T_{e}^2} \EE \norm{z^* - z_{(e)}}^2 
+ \frac{768 \sigma^2}{\mu^2 T_{e}},~   \eta_t = \frac{2}{\mu (t + t_0 + 1)},
~ t_0 = 4 \ceil{\frac{L}{\mu}},
\end{align*}
and output $z_{(e)} \coloneqq z_{T_e}$. 
Then for any $z_{(0)} \in X \times Y$, by choosing $E = \log_2 \rbr{
\frac{\norm{z_{(0)} - z^*}^2}{\epsilon}
}$, we have 
\begin{align*}
\EE \sbr{
 \norm{ z_{(E)} - z^* }^2  } \leq \epsilon.
\end{align*}
The total iteration and sample complexities can be bounded by
\begin{align*}
\cO \rbr{
 \frac{L}{\mu}  \log_2 \rbr{\frac{\norm{z^* - z_0}}{\epsilon}}
 + \frac{\sigma^2}{\mu^2 \epsilon} 
 }.
\end{align*}
\end{theorem}

\begin{proof}
In view of Theorem \ref{thrm_seg_sublinear_deterministic_err}, we obtain  
\begin{align*}
& \EE \sbr{
 \norm{z^* - z_{(e+1)}}^2  } \leq 
\frac{6  t_0^2}{T_{e}^2} \EE \norm{z^* - z_{(e)}}^2 
+ \frac{768 \sigma^2}{\mu^2 T_{e}}.
\end{align*}
It can be readily seen that from an induction argument, by choosing 
$
T_{e} = 6 t_0 +   \frac{768 \cdot 2^{e+3} \sigma^2}{\mu^2 \norm{z^* - z_0}^2},
$
we have 
\begin{align*}
\EE \sbr{
 \norm{z^* - z_{(e)}}^2  } \leq 2^{-e}  \norm{z^* - z_{(0)}}^2, \forall e \geq 0.
\end{align*}
The total number of iterations to obtain $\EE \sbr{
 \norm{z^* - z_{(e)}}^2  } \leq \epsilon$ is thus bounded by 
 \begin{align*}
 \cO \rbr{
 t_0 \log_2 \rbr{\frac{\norm{z^* - z_0}}{\epsilon}}
 + \frac{\sigma^2}{\mu^2 \epsilon} 
 }
 = 
  \cO \rbr{
 \frac{L}{\mu}  \log_2 \rbr{\frac{\norm{z^* - z_0}}{\epsilon}}
 + \frac{\sigma^2}{\mu^2 \epsilon} 
 }.
 \end{align*}
 The proof is then completed.
\end{proof}

A few remarks are in order before we conclude  this section. 
In view of Theorem \ref{theorem_stoch_smeg}, the proposed REG method obtains the optimal iteration and sample complexities for solving smooth and strongly monotone variational inequalities. 
It should be noted that similar performances have been obtained recently  in \cite{kotsalis2022simple} and \cite{beznosikov2021distributed}. 
Nevertheless, given our discussions in Section \ref{sec_intro}, these methods are  clearly non-optimal for our problem of interest \eqref{def_problem} even in the deterministic setting. 
In following sections, we show that the error condition associated with the REG method makes it natural candidate for the catalyst scheme to be developed for  \eqref{def_problem}.
Consequently, one can accelerate the REG method and obtain optimal iteration and sample complexities up to logarithmic factors.


\section{Deterministic  Minimax Catalyst Scheme}\label{sec_minimax_deterministic}
In this section, we formally introduce the catalyst scheme for the minimax problem \eqref{def_problem}.
Compared to Section \ref{sec_eg}, we no longer require the problem to be strongly-convex-strongly-concave.
Instead, going forward we only assume that \eqref{def_problem} is always strongly-concave w.r.t. $y$.  
It should be noted that when $\mu_d = 0$ and subsequently $\mu_p = 0$ (since $\mu_d \geq \mu_p$ given our discussion in Section \ref{sec_intro}), one can simply add a proper strongly-convex-strongly-concave perturbation before applying  the catalyst scheme.
As will be clarified by our ensuing discussion, this simple reduction technique again would yield optimal iteration and sample complexities up to logarithmic factors.

The proposed catalyst scheme for minimax problem (\texttt{Catalyst-Minimax}) is presented in Algorithm \ref{alg_catalyst_deterministic}.
Similar to \eqref{eq:cp_subproblem}, at the $k$-th iteration, 
for a given prox-center $\hat x_k$ and some $\beta_k \ge 0$, we 
denote 
\begin{align} \label{eq:def_Phi_sd}
 \Phi_k(x, y) := F(x, y) + \tfrac{\beta_k}{2} \|x - \hat x_k\|^2,
\end{align}
and define the proximal subproblem as 
\begin{align} \label{eq:cp_subproblem_minmax}
\min_{x \in X} \left\{ \phi_k(x) := \max_{y \in Y} \Phi_k(x, y) \right\}.
\end{align}
We assume the catalyst scheme has access to a  to-be-catalyzed method $\cA$, which can find a pair of solution $(\tilde z_k, z_k) \in X \times X$ for \eqref{eq:cp_subproblem_minmax} such that 
\begin{align} \label{eq:cp_inexact_sd}
 \Phi_k(\tilde x_k, \tilde y) - \Phi_k(\tilde x, \tilde y_k) + \tfrac{\alpha_k}{2}  \|\tilde z - z_k\|^2  
\le  \tfrac{\varepsilon_k}{2} \|\tilde z - z_k^0 \|^2 , 
\end{align}
for any $\tilde z \in Z$, where $z_k^0 = (x_k^0, y_k^0)$
is a given pair of starting points for method $\cA$.

\begin{algorithm}[H]
\caption{\texttt{Catalyst-Minimax}$(\cA)$: catalyst scheme for minimax optimization}
\begin{algorithmic}\label{alg_catalyst_deterministic}
\STATE{ {\bf Input:} initial points $\bar z_0 = \tilde z_0 = z_0$, and algorithmic parameters $\{\alpha_k\}$, $\{\beta_k\}$, $\{\gamma_k\}$, and $\{\varepsilon_k\}$.}
\FOR{$k =1, 2, \ldots,$}
\STATE{
\vspace{-0.2in}
\begin{align}
\hat x_k &= \gamma_k \bar x_{k-1} + (1-\gamma_k) \tilde x_{k-1}.  \label{eq:define_hat_x_sd}\\
 (\tilde z_k, z_k) &  = \cA(\Phi_k) ~  \text{s.t. \eqref{eq:cp_inexact_sd} holds with $x_k^0 = \hat x_k$ and $y_k^0=y_{k-1}$.} \label{eq:define_init_point} \\
\bar x_{k} &= 
 \tfrac{1}{\alpha_k \gamma_k  +  \mu_p (1-\gamma_k) }\left[\alpha_k x_k + ( \mu_p -  \alpha_k) (1-\gamma_k) \tilde x_{k-1}\right]. \label{eq:define_bar_x_sd}
\end{align}
}
\ENDFOR
\end{algorithmic} \label{alg:basic_cat_sd}
\end{algorithm}

It should be noted that although both $\alpha_k$ and $\varepsilon_k$ are involved in defining the approximate update \eqref{eq:cp_inexact_sd}, the catalyst scheme only requires $\alpha_k$ to perform subsequent updates \eqref{eq:define_hat_x_sd} and \eqref{eq:define_bar_x_sd}. 
Clearly, when $\alpha_k = \beta_k$ and $\varepsilon_k  = 0$,  the proximal step \eqref{eq:def_Phi_sd} is computed exactly, and the catalyst scheme reduces to the accelerated proximal point method applied to \eqref{def_problem}. 
Consequently, \eqref{eq:cp_inexact_sd} can be viewed as characterizing the approximate update of the proximal step.
For the remainder of our discussion in this section, we establish the convergence of the catalyst scheme with proper conditions imposed on \eqref{eq:cp_inexact_sd}, and subsequently invoke the REG method developed in Section \ref{sec_eg} to certify such conditions. 
As our initial step, we first make the following technical observation that relates the distance on the dual space with the primal optimality gap.

\begin{lemma}\label{lemma_dual_dist_to_primal}
Let $(x^*,y^*)$ be a pair of optimal primal and dual solutions of \eqref{def_problem}.
Define
\begin{align} \label{eq:def_yk}
\tilde y_k^* := \argmax_{y \in Y} \Phi_k(\tilde x_k, y).
\end{align}
Then
\begin{align*}
 \tfrac{\mu_d}{2} \|y^* - \tilde y_k^* \|^2 \le f(\tilde x_k) - f(x^*).
\end{align*}
\end{lemma}

\begin{proof}
By definition, we have
\begin{align*}
\tilde y_k^* =  \argmax_{y \in Y} \Phi_k(\tilde x_k, y) = \argmax_{y \in Y} F(\tilde x_k, y).
\end{align*}
Given the optimality condition of $\tilde y_k^*$, we further have
\begin{align*}
f(\tilde x_k) = F(\tilde x_k, \tilde y_k^*) \ge F(\tilde x_k, y^*) + \tfrac{\mu_d}{2} \|y^* - \tilde y_k^*\|^2 
\geq F(x^*, y^*) + \tfrac{\mu_d}{2} \|y^* - \tilde y_k^*\|^2
= f(x^*) + \tfrac{\mu_d}{2} \|y^* - \tilde y_k^*\|^2.
\end{align*}
The proof is then completed.
\end{proof}

With Lemma \ref{lemma_dual_dist_to_primal} in place, we proceed to characterize each step of the proposed catalyst scheme. 

\begin{lemma}\label{lemma_recursion_minimax_catalyst_one_step_deterministic}
Within Algorithm~\ref{alg:basic_cat_sd}, 
for any $k \ge 1$, we have
\begin{align}
&\rbr{1-\tfrac{4 \varepsilon_k}{\mu_d}} \sbr{ f(\tilde x_k) - f(x^*)}
+ \tfrac{\alpha_k \gamma_k^2 + \gamma_k(1-\gamma_k) \mu_p}{2} \|x^* -  \bar x_k\|^2 +  \tfrac{\alpha_k}{2}  \|\tilde y_k^* - y_k\|^2 \nonumber  \\
\leq & \rbr{1-\gamma_k + \tfrac{4 \varepsilon_k}{\mu_d}} \sbr{ f(\tilde x_{k-1}) - f(x^*)} +
 \tfrac{(\beta_k + \varepsilon_k) \gamma_k^2}{2} \|x^* - \bar x_{k-1}  \|^2 + \varepsilon_k \|\tilde y_{k-1}^* - y_{k-1}\|^2  . \label{minimax_outer_one_step_recursion}
\end{align}
\end{lemma}

\begin{proof}
By \eqref{eq:cp_inexact_sd} and the selection of initial points in \eqref{eq:define_init_point}, we have
\begin{align*}
  \Phi_k(\tilde x_k, \tilde y) - \Phi_k(\tilde x, \tilde y_k) + \tfrac{\alpha_k}{2}  \|\tilde z - z_k\|^2  
\le  \tfrac{\varepsilon_k}{2} [ \|\tilde x - \hat x_k \|^2 + \|\tilde y - y_{k-1}\|^2] ,  ~ \forall \tilde z \in Z,
\end{align*} 
which, in view of the definition of $\Phi_k$ in \eqref{eq:def_Phi_sd},  implies that
\begin{align*}
  F(\tilde x_k, \tilde y) - F(\tilde x, \tilde y_k)  
 + \tfrac{\alpha_k}{2}  \|\tilde z - z_k\|^2  
\le \tfrac{\beta_k + \varepsilon_k}{2}  \|\tilde x - \hat x_k \|^2 +\tfrac{\varepsilon_k}{2} \|\tilde y - y_{k-1}\|^2
, ~ \forall \tilde z \in Z.
\end{align*}
Using the above inequality and the fact that $F(\tilde x, \tilde y_k) \le f(\tilde x)$, we obtain
\begin{align*}
  F(\tilde x_k, \tilde y) - f(\tilde x) 
 + \tfrac{\alpha_k}{2}  \|\tilde z - z_k\|^2  
\le \tfrac{\beta_k + \varepsilon_k}{2} \|\tilde x - \hat x_k \|^2 + \tfrac{\varepsilon_k}{2} \|\tilde y - y_{k-1}\|^2 , ~ \forall \tilde z \in Z.
\end{align*}
Setting $\tilde y = \tilde y_k^*$
in the above relation yields
\begin{align*}
f(\tilde x_k) - f(\tilde x) 
+  \tfrac{\alpha_k}{2}  [\|\tilde x - x_k\|^2  + \|\tilde y_k^* - y_k\|^2]
\le  \tfrac{\beta_k + \varepsilon_k}{2} \|\tilde x - \hat x_k \|^2 +\tfrac{\varepsilon_k}{2} \|\tilde y_k^* - y_{k-1}\|^2
\end{align*}
for any $\tilde x \in X$. 
In addition, we also have
\begin{align*}
\|\tilde y_k^* - y_{k-1}\|^2 &\le 2 \|\tilde y_k^* - \tilde y_{k-1}^*\|^2 + 2 \|\tilde y_{k-1}^* - y_{k-1}\|^2\\
&\le   4 (\|\tilde y_k^* - y^*\|^2 + \| \tilde y_{k-1}^* - y^*\|^2) + 2 \|\tilde y_{k-1}^* - y_{k-1}\|^2\\
&\le \tfrac{8}{\mu_d} \left [f(\tilde x_k) - f(x^*) + f(\tilde x_{k-1}) - f(x^*)\right] + 2 \|\tilde y_{k-1}^* - y_{k-1}\|^2.
\end{align*}
Combining the previous two relations, we have
\begin{align*}
&f(\tilde x_k) - f(\tilde x)  
+  \tfrac{\alpha_k}{2}  [\|\tilde x - x_k\|^2  + \|\tilde y_k^* - y_k\|^2] \\
\leq &  \tfrac{\beta_k + \varepsilon_k}{2} \|\tilde x - \hat x_k \|^2 
+  \tfrac{4 \varepsilon_k}{\mu_d} \left [f(\tilde x_k) - f(x^*) + f(\tilde x_{k-1}) - f(x^*)\right] + \varepsilon_k \|\tilde y_{k-1}^* - y_{k-1}\|^2.
\end{align*}
Setting $\tilde x = \gamma_k x^* + (1 - \gamma_k) \tilde x_{k-1}$ and using the (strong) convexity
of $f$, we obtain
\begin{align*}
&f(\tilde x_k) - \gamma_k f(x^*) - (1-\gamma_k) f( \tilde x_{k-1})
+ \tfrac{\mu_p \gamma_k (1 -\gamma_k)}{2} \|x^* -  \tilde x_{k-1}\|^2 +  \tfrac{\alpha_k}{2}  \|\tilde y_k^* - y_k\|^2 \\
& ~~~ +  \tfrac{\alpha_k}{2}  [\|\gamma_k x^* + (1 - \gamma_k) \tilde x_{k-1} - x_k\|^2  + \|\tilde y_k^* - y_k\|^2] \\
\leq &  \tfrac{\beta_k + \varepsilon_k}{2} \|\gamma_k x^* + (1 - \gamma_k) \tilde x_{k-1} - \hat x_k \|^2 
+  \tfrac{4 \varepsilon_k}{\mu_d} \left [f(\tilde x_k) - f(x^*) + f(\tilde x_{k-1}) - f(x^*)\right] + \varepsilon_k \|\tilde y_{k-1}^* - y_{k-1}\|^2,
\end{align*}
which, after substituting the definition of $\hat{x}_k$ in \eqref{eq:define_hat_x_sd}, becomes  
\begin{align*}
& \rbr{1 - \frac{4 \varepsilon_k}{\mu_d}} 
\sbr{f(\tilde{x}_k) - f(x^*)} 
+  \tfrac{\alpha_k}{2}  \|\gamma_k x^* + (1 - \gamma_k) \tilde x_{k-1} - x_k\|^2 
+  \tfrac{\mu_p \gamma_k (1 -\gamma_k)}{2} \|x^* -  \tilde x_{k-1}\|^2 
+ \tfrac{\alpha_k}{2} \norm{\tilde{y}_k^* - y_k}^2 \\
\leq & 
(1-\gamma_k + \frac{4 \varepsilon_k}{\mu_d}) \sbr{f(\tilde{x}_{k-1}) - f(x^*)} 
+ \tfrac{(\beta_k + \varepsilon_k) \gamma_k^2}{2} \norm{x^* - \overline{x}_{k-1}}^2 
+ \varepsilon_k \|\tilde y_{k-1}^* - y_{k-1}\|^2.
\end{align*}
It remains to make use of observation \eqref{min_aggregate_x} again, which simplifies the above inequality into 
\begin{align*}
& \rbr{1 - \frac{4 \varepsilon_k}{\mu_d}} 
\sbr{f(\tilde{x}_k) - f(x^*)} 
+ \frac{\alpha_k \gamma_k^2 + \gamma_k(1-\gamma_k) \mu_p}{2} \norm{x^* - \overline{x}_k}^2  
+ \tfrac{\alpha_k}{2} \norm{\tilde{y}_k^* - y_k}^2 \\
\leq & 
\rbr{ 1-\gamma_k + \frac{4 \varepsilon_k}{\mu_d}} \sbr{f(\tilde{x}_{k-1}) - f(x^*)} 
+ \tfrac{(\beta_k + \varepsilon_k) \gamma_k^2}{2} \norm{x^* - \overline{x}_{k-1}}^2 
+ \varepsilon_k \|\tilde y_{k-1}^* - y_{k-1}\|^2.
\end{align*} 
The proof is then completed.
\end{proof}

We now turn our attention to the global convergence of the proposed catalyst scheme under proper requirements of the approximate proximal update \eqref{eq:cp_inexact_sd}.

\begin{lemma}\label{lemma_minimax_outer_recursion_raw}
Suppose \eqref{eq:cp_inexact_sd} holds, and 
\begin{align}
\frac{\alpha_k \gamma_k^2}{1 - 4 \varepsilon_k / \mu_d}
& \geq \frac{(\beta_{k+1} + \varepsilon_{k+1}) \gamma_{k+1}^2 }{1 - \gamma_{k+1} + 4 \varepsilon_{k+1} / \mu_d}, \label{minimax_outer_param_choice_11} \\
 \frac{\varepsilon_{k+1}}{\alpha_k}  & \leq \frac{1 - \gamma_{k+1} + 4 \varepsilon_{k+1} / \mu_d}{2 (1 - 4 \varepsilon_k / \mu_d)} .\label{minimax_outer_param_choice_12}
\end{align}
Define 
\begin{align}\label{def_Gamma_minimax_deterministic}
\Gamma_k = 
\begin{cases}
1, ~ & k = 1;  \\
\Gamma_{k-1} \frac{1 - \gamma_k + 4 \varepsilon_k / \mu_d}{1 - 4 \varepsilon_k / \mu_d}, ~ & k \geq 2.
\end{cases}
\end{align}
Then we have 
\begin{align*}
& {f(\tilde{x}_k) - f(x^*)} + \frac{\alpha_k}{2 (1-4 \varepsilon_k / \mu_d) } \norm{\tilde{y}_k^* - y_k}^2 \\
\leq  &
\Gamma_k \sbr{ \frac{1-\gamma_1 + 4\varepsilon_1 / \mu_d}{(1 - 4 \varepsilon_1/\mu_d)}  \sbr{f(\tilde{x}_0) - f(x^*)}  + \frac{(\beta_1 + \varepsilon_1) \gamma_1^2}{2 (1 - {4 \varepsilon_1}/{\mu_d}) }  \norm{x - \tilde{x}_0}^2   + 
\frac{\varepsilon_1}{(1 - {4 \varepsilon_1}/{\mu_d}) } \norm{\tilde{y}_0^* - y_0}^2
}.
\end{align*}
\end{lemma}

\begin{proof}
For notational convenience, let us write 
$\tilde{\Delta}_k = f(\tilde{x}_k) - f(x^*)$, 
$d_Y^k = \norm{ x^* -   \overline{x}_k}^2$, and
$d_Y^k = \norm{\tilde{y}_k^* - y_k}^2$.
Taking $\mu_p = 0$ in \eqref{minimax_outer_one_step_recursion}, and further dividing both sides by 
$(1 - \frac{4 \varepsilon_k}{\mu_d}) \Gamma_k$, we obtain 
\begin{align}
& \frac{1}{\Gamma_k} \tilde{\Delta}_k + \frac{\alpha_k \gamma_k^2}{2 (1 - {4 \varepsilon_k}/{\mu_d}) \Gamma_k } d_X^k + \frac{\alpha_k}{2 (1 - {4 \varepsilon_k}/{\mu_d}) \Gamma_k} d_Y^k  \nonumber  \\
\leq & 
\frac{1}{\Gamma_{k-1}} \tilde{\Delta}_{k-1} + \frac{(\beta_k + \varepsilon_k) \gamma_k^2}{2 (1 - {4 \varepsilon_k}/{\mu_d}) \Gamma_k} d^{k-1}_X + 
\frac{\varepsilon_k}{(1 - {4 \varepsilon_k}/{\mu_d}) \Gamma_k} d^{k-1}_Y ,~ k \geq 2,  \label{minimax_outer_telescope_1}
\end{align}
where the last inequality follows from the definition of $\cbr{\Gamma_k}$.
In addition, we also have 
\begin{align}
& \frac{1}{\Gamma_1} \tilde{\Delta}_1 + \frac{\alpha_1 \gamma_1^2}{2 (1 - {4 \varepsilon_1}/{\mu_d}) \Gamma_1 } d_X^1 + \frac{\alpha_1}{2 (1 - {4 \varepsilon_1}/{\mu_d}) \Gamma_1} d_Y^1 \nonumber  \\
\leq & 
\frac{1-\gamma_1 + 4\varepsilon_1 / \mu_d}{(1 - 4 \varepsilon_1/\mu_d)\Gamma_1}  \tilde{\Delta}_{0} + \frac{(\beta_1 + \varepsilon_1) \gamma_1^2}{2 (1 - {4 \varepsilon_1}/{\mu_d}) \Gamma_1} d^{0}_X + 
\frac{\varepsilon_1}{(1 - {4 \varepsilon_1}/{\mu_d}) \Gamma_1} d^{0}_Y . \label{minimax_outer_telescope_2}
\end{align}
In view of \eqref{minimax_outer_param_choice_11}, \eqref{minimax_outer_param_choice_12} and the definition of $\cbr{\Gamma_k}$, we have 
\begin{align*}
\frac{\alpha_k \gamma_k^2}{(1 - 4 \varepsilon_k / \mu_d) \Gamma_k} \geq \frac{(\beta_{k+1} + \varepsilon_{k+1}) \gamma_{k+1}^2}{(1 - 4 \varepsilon_{k+1} / \mu_d) \Gamma_{k+1}},
~
\frac{\alpha_k}{2(1 - 4 \varepsilon_k / \mu_d) \Gamma_k} 
\geq \frac{ \varepsilon_{k+1} }{(1- 4\varepsilon_{k+1} / \mu_d) \Gamma_{k+1}} .
\end{align*}
Taking the telescopic sum of \eqref{minimax_outer_telescope_1} and \eqref{minimax_outer_telescope_2}  yields 
\begin{align*}
\frac{1}{\Gamma_k} \tilde{\Delta}_k + \frac{\alpha_k}{2 (1-4 \varepsilon_k / \mu_d) \Gamma_k} d^k_Y
\leq 
\frac{1-\gamma_1 + 4\varepsilon_1 / \mu_d}{(1 - 4 \varepsilon_1/\mu_d)\Gamma_1}  \tilde{\Delta}_{0} + \frac{(\beta_1 + \varepsilon_1) \gamma_1^2}{2 (1 - {4 \varepsilon_1}/{\mu_d}) \Gamma_1} d^{0}_X + 
\frac{\varepsilon_1}{(1 - {4 \varepsilon_1}/{\mu_d}) \Gamma_1} d^{0}_Y.
\end{align*}
The desired claims follows immediately by noting that $z_0 = \tilde{z}_0 = \overline{z}_0$.
\end{proof}

We are now ready to establish the convergence of the catalyst scheme for convex-strongly-concave problems ($\mu_p = 0$).

\begin{lemma}\label{lemma_general_convergence_minimax_deterministic_nsc}
Suppose $\mu_p = 0$ for \eqref{def_problem}. Fix total iterations $K \geq 1$ a priori. 
Choose 
\begin{align*}
\gamma_k = \frac{2}{k+1},  
~ 
\beta_k =   \frac{ \mu_d (k+1)}{2(k+2)}.
\end{align*}
In addition, suppose $\alpha_k$ is chosen such that there exists  $\varepsilon_k$ certifying \eqref{eq:cp_inexact_sd} with 
\begin{align}\label{general_minimax_determinsitic_alpha_beta}
\alpha_k = \beta_k (1 + \varepsilon), ~ \varepsilon_k = \beta_k \varepsilon,  
\end{align}
for some 
\begin{align}\label{general_minimax_determinsitic_varepsilon}
\varepsilon \leq \min \cbr{\frac{1}{12}, \frac{1}{(K+1) (K+2)}, \frac{\norm{x^* - \tilde{x}_0}^2}{2 \sbr{f(\tilde{x}_0) - f(x^*) }}}. 
\end{align}
Then we have 
\begin{align*}
 {f(\tilde{x}_K) - f(x^*)} +  \frac{\mu_d}{6} \norm{\tilde{y}_K^* - y_K}^2 
\leq  
\frac{12}{K^2}  \sbr{   2 \mu_d \norm{x^* - \tilde{x}_0}^2   + 
\mu_d \norm{\tilde{y}_0^* - y_0}^2
}.
\end{align*}
\end{lemma}

\begin{proof}
Given the choice of $\varepsilon_k$, $\beta_k$ and $\varepsilon \leq \frac{1}{(K+1)(K+2)}$,  it can be readily verified that $\frac{4 \varepsilon_{k+1}}{\mu_d} \leq \frac{2}{(k+1)(k+2)}$ for $k \leq K$.
Combining this with the definition of $\gamma_k = \frac{2}{k+1}$,  it holds that
\begin{align}\label{deterministic_minimax_gamma_ratio_bound}
\frac{1 - \gamma_{k+1} + 4 \varepsilon_{k+1} / \mu_d}{1 - 4 \varepsilon_k / \mu_d} \in \sbr{ \frac{k}{k+2}, \frac{k+1}{k+3}}.
\end{align}
Since $\varepsilon_{k+1} \geq \varepsilon_k$ and $\gamma_k \geq \gamma_{k+1}$, the above relation implies
\begin{align*}
\frac{1 - \gamma_{k} + 4 \varepsilon_{k} / \mu_d}{1 - 4 \varepsilon_k / \mu_d} \leq \frac{k+1}{k+3},
\end{align*}
and consequently $\Gamma_k \leq \frac{12}{k^2}$.
It remains to verify that \eqref{minimax_outer_param_choice_11} and \eqref{minimax_outer_param_choice_12}  in Lemma \ref{lemma_minimax_outer_recursion_raw} hold.
In particular, \eqref{minimax_outer_param_choice_11} is a direct consequence of \eqref{deterministic_minimax_gamma_ratio_bound} and the choice of $\alpha_k = \beta_k (1 + \varepsilon)$, $\varepsilon_k = \beta_k \varepsilon$, $\gamma_k = \frac{2}{k+1}$,  and $\beta_k = \frac{\mu_d (k+1)}{2 (k+2)}$. 
Additionally, \eqref{minimax_outer_param_choice_12} follows from \eqref{deterministic_minimax_gamma_ratio_bound} and choice of $\alpha_k = \beta_k (1 + \varepsilon)$, $\beta_k = \frac{\mu_d (k+1)}{2 (k+2)}$,   $\varepsilon_k = \beta_k \varepsilon$, and $\varepsilon \leq 1/12$.

Combining the above observations, we can now invoke Lemma \ref{lemma_minimax_outer_recursion_raw} and obtain 
\begin{align}
& {f(\tilde{x}_K) - f(x^*)} + \frac{\mu_d}{12} \norm{\tilde{y}_K^* - y_K}^2 \nonumber \\
\leq & {f(\tilde{x}_K) - f(x^*)} + \frac{\alpha_K}{2 (1-4 \varepsilon_K / \mu_d) } \norm{\tilde{y}_K^* - y_K}^2 \nonumber \\
\leq  &
\frac{12}{K^2} \sbr{ \frac{1-\gamma_1 + 4\varepsilon_1 / \mu_d}{(1 - 4 \varepsilon_1/\mu_d)}  \sbr{f({x}_0) - f(x^*)}  + \frac{(\beta_1 + \varepsilon_1) \gamma_1^2}{2 (1 - {4 \varepsilon_1}/{\mu_d}) }  \norm{x - {x}_0}^2   + 
\frac{\varepsilon_1}{(1 - {4 \varepsilon_1}/{\mu_d}) } \norm{\tilde{y}_0^* - y_0}^2
},
\label{ineq_nonsc_minimax_deterministic_raw}
\end{align}
where the first inequality follows from 
\begin{align*}
\frac{\alpha_k}{2 (1 - 4 \varepsilon_k / \mu_d)} = \frac{\beta_k (1 + \varepsilon) }{2 (1 - 4 \beta_k \varepsilon / \mu_d)}   \geq \frac{13 \mu_d / 36}{2 (1 - 1/9)} \geq \frac{\mu_d}{6}.
\end{align*}
 Now simplifying \eqref{ineq_nonsc_minimax_deterministic_raw} after plugging in the choice of $(\gamma_1, \varepsilon_1, \beta_1)$ and applying $\varepsilon \leq 1/12$ again,
 we obtain 
 \begin{align}\label{ineq_deterministic_minimax_smooth_wo_cond}
 {f(\tilde{x}_K) - f(x^*)} +  \frac{\mu_d}{6} \norm{\tilde{y}_K^* - y_K}^2 
\leq  
\frac{12}{K^2}  \sbr{   2 \varepsilon  \sbr{f(\tilde{x}_0) - f(x^*)}  +  \mu_d \norm{x^* - \tilde{x}_0}^2   + 
\mu_d \norm{\tilde{y}_0^* - y_0}^2
}.
\end{align}
The desired claim then follows from the choice of $\varepsilon \leq \frac{\norm{x^* - \tilde{x}_0}^2}{2 \sbr{f(\tilde{x}_0) - f(x^*) }}$.
\end{proof}

It can be seen from \eqref{general_minimax_determinsitic_varepsilon} of  Lemma \ref{lemma_general_convergence_minimax_deterministic_nsc} that the convergence of the catalyst scheme requires the knowledge of  the unknown quantity ${\norm{x^* - \tilde{x}_0}^2}/{\rbr{f(\tilde{x}_0) - f(x^*) }}$.
As will be clarified in Theorem \ref{thrm_minimax_catalyst_deterministic_nsc}, 
when the proximal step \eqref{eq:cp_subproblem_minmax} is approximately solved by the REG method introduced in Section \ref{sec_eg}, 
one can simply use an underestimate of this quantity, and the total computational complexity of the catalyst scheme will only increase by a logarithmic factor as the price of potential underestimation. 

For strongly-convex-strongly-concave problems, we proceed to describe a simple restarting procedure applicable to the proposed catalyst scheme,
and establish its convergence properties.

\begin{algorithm}[H]
\caption{\texttt{R-Catalyst-Minimax}$(\cA)$: restarting catalyst for strongly-convex-strongly-concave problems}
\begin{algorithmic}
\STATE{ {\bf Input:} initial point $z_{(0)}$, to-be-catalyzed method $\cA$, total number of epochs $E$, epoch length $\cbr{K_e}$.}
\FOR{epoch $e = 1, 2, \ldots, E$}
\STATE{
Let $z_{(e)}$ be the output of  running $\texttt{Catalyst-Minimax} (\cA)$ starting from $z_{(e-1)}$ for  $K_e$ iterations. 
}
\ENDFOR
\end{algorithmic} \label{alg_restart_opt_catalyst_minimax}
\end{algorithm}

\begin{lemma}\label{lemma_deterministic_minimax_sc}
Suppose $\mu_p > 0$.
Within each epoch of the \texttt{R-Catalyst-Minimax} scheme, choose
\begin{align*}
K_e \equiv K \geq 12 \sqrt{\frac{\mu_d}{\mu_p}}, ~ 
\gamma_k = \frac{2}{k+1},  
~ 
\beta_k =   \frac{ \mu_d (k+1)}{2(k+2)},
\end{align*}
and suppose $\alpha_k$ is chosen such that there exists  $\varepsilon_k$ certifying \eqref{eq:cp_inexact_sd} with 
\begin{align*}
\alpha_k = \beta_k (1 + \varepsilon), ~ \varepsilon_k = \beta_k \varepsilon,  
\end{align*}
for some 
$
\varepsilon \leq \min \cbr{\frac{1}{12}, \frac{1}{(K+1) (K+2)}}.
$
Then
\begin{align*}
 {f(\tilde{x}_{(e)}) - f(x^*)} +  \frac{\mu_d}{6} \norm{\tilde{y}_{(e)}^* - y_{(e)}}^2 
 \leq 
 \rbr{\frac{1}{2} }^e
 \sbr{
 {f(\tilde{x}_{(0)}) - f(x^*)} +  \frac{\mu_d}{6} \norm{\tilde{y}_{(0)}^* - y_{(0)}}^2 
 }.
\end{align*}
\end{lemma}

\begin{proof}
Given the choice of parameters $\cbr{(\alpha_k, \beta_k, \gamma_k)}$ and $\varepsilon \leq \frac{1}{12}$, one can apply \eqref{ineq_deterministic_minimax_smooth_wo_cond} in Lemma \ref{lemma_general_convergence_minimax_deterministic_nsc} and obtain 
\begin{align*}
& {f(\tilde{x}_{(e+1)}) - f(x^*)} +  \frac{\mu_d}{6} \norm{\tilde{y}_{(e+1)}^* - y_{(e+1)}}^2  \\
\leq  &  
\frac{12}{K^2}  \sbr{   f(\tilde{x}_{(e)}) - f(x^*) +  \mu_d \norm{x^* - \tilde{x}_{(e)}}^2   + 
\mu_d \norm{\tilde{y}_{(e)}^* - y_{(e)}}^2
}  \\
\overset{(a)}{\leq} & 
\frac{36 \mu_d}{K^2 \mu_p} \sbr{ f(\tilde{x}_{(e)}) - f(x^*) }
+ \frac{12 \mu_d}{K^2} \norm{\tilde{y}_{(e)}^* - y_{(e)}}^2 \\
\overset{(b)}{\leq} & 
\frac{1}{2} \sbr{
 {f(\tilde{x}_{(e)}) - f(x^*)} +  \frac{\mu_d}{6} \norm{\tilde{y}_{(e)}^* - y_{(e)}}^2
},
\end{align*}
where $(a)$ follows from $f(\cdot)$ being $\mu_p$ strongly-convex, 
and $(b)$ follows from the choice of $K$.
Applying the above relation recursively yields the desired claim.
\end{proof}

Until now, the established convergence of the proposed \texttt{Catalyst-Minimax} scheme hinges upon the approximate proximal step \eqref{eq:cp_inexact_sd} to satisfy certain error conditions introduced in Lemma \ref{lemma_general_convergence_minimax_deterministic_nsc} and \ref{lemma_deterministic_minimax_sc}.
In the following, we show that such error conditions can be naturally satisfied when the proximal step \eqref{eq:def_Phi_sd} is solved by the REG method introduced in Section \ref{sec_eg} with proper parameter specifications. 
Consequently, we are able to establish the total iteration complexity of the proposed minimax catalyst scheme when catalyzing the REG method. 
As before, we first proceed to the case where $\mu_p = 0$.

\begin{theorem}\label{thrm_minimax_catalyst_deterministic_nsc}
Suppose $\mu_p = 0$. 
For any $\epsilon > 0$, run \texttt{Catalyst-Minimax}(\texttt{REG}) with 
\begin{align*}
K \geq \sqrt{\frac{4 \sbr{   2 \mu_d \norm{x^* - \tilde{x}_0}^2   + 
\mu_d \norm{\tilde{y}_0^* - y_0}^2
}}{\epsilon}}, ~
\gamma_k = \frac{2}{k+1},  
~ 
\beta_k =   \frac{ \mu_d (k+1)}{2(k+2)},
~
\alpha_k =  \frac{ \beta_k \Lambda_T}{\Lambda_T - \Lambda_0},
\end{align*}
where 
\begin{align*}
T \geq \frac{6 L \log(12) }{\mu_d} + \frac{6 L \log (6K^2)  }{\mu_d} +   \frac{6 L \log (2 \sbr{f(\tilde{x}_0) - f(x^*)} / \norm{x^* - \tilde{x}_0}^2)  }{\mu_d} , ~ \Lambda_t = (1 + \frac{\mu_d}{6 L })^t.
\end{align*}
At the $k$-th  iteration of the catalyst scheme, set $(\tilde{z}_k, z_k)$ by the output of \texttt{REG}($\Phi_k$), initialized at $(\hat{x}_k, y_{k-1})$ and running for a total of $T$ steps  with   
$
\eta_t =  \frac{(k+2)}{3(k+1) L} .
$
Then we obtain 
\begin{align*}
 {f(\tilde{x}_K) - f(x^*)} +  \frac{\mu_d}{6} \norm{\tilde{y}_K^* - y_K}^2  \leq \epsilon.
\end{align*}
\end{theorem}

\begin{proof}
Given the choice of $\beta_k$, it can be readily seen that the subproblem \eqref{eq:def_Phi_sd} solved by \texttt{REG} is $\beta_k$ strongly-convex-strongly-concave, and $2L$ smooth. 
Combining this observation with the choice of $\cbr{\eta_t}$, 
one can apply Lemma \ref{lemma_eg_deterministic} and obtain that  \eqref{general_minimax_determinsitic_alpha_beta} is satisfied with
$\varepsilon = \frac{1}{\Lambda_T - 1}$ and $\Lambda_t = (1 + \frac{\mu_d}{6L})^t$.
It remains to make proper choice of $T$ so that  \eqref{general_minimax_determinsitic_varepsilon} holds. 
From \eqref{ineq_eg_deterministic_general} Lemma \ref{lemma_eg_deterministic} this can be readily satisfied by taking 
\begin{align*}
T \geq  \frac{6 L \log(12) }{\mu_d} + \frac{6 L \log (6K^2)  }{\mu_d} +  \frac{6 L \log (2 \sbr{f(\tilde{x}_0) - f(x^*)} / \norm{x^* - \tilde{x}_0}^2)  }{\mu_d}.
\end{align*}
Consequently, one can then invoke Lemma \ref{lemma_general_convergence_minimax_deterministic_nsc} and obtain 
\begin{align*}
 {f(\tilde{x}_K) - f(x^*)} +  \frac{\mu_d}{6} \norm{\tilde{y}_K^* - y_K}^2 
\leq  
\frac{12}{K^2}  \sbr{  2  \mu_d \norm{x^* - \tilde{x}_0}^2   + 
\mu_d \norm{\tilde{y}_0^* - y_0}^2
} \leq \epsilon,
\end{align*}
where the last inequality follows from the choice of $K$.
\end{proof}

Similarly, we proceed to instantiate the \texttt{R-Catalyst-Minimax} scheme with the proposed REG method, and establish its total iteration complexity.

\begin{theorem}\label{thrm_minimax_deterministic_iter_complexity}
Suppose $\mu_p > 0$. 
Run \texttt{R-Catalyst-Minimax}(\texttt{REG}) with epoch length $K \geq 12 \sqrt{\frac{\mu_d}{\mu_p}}$.
Within each epoch, choose 
\begin{align*}
\gamma_k = \frac{2}{k+1},  
~ 
\beta_k =   \frac{ \mu_d (k+1)}{2(k+2)},
~
\alpha_k =  \frac{ \beta_k \Lambda_T}{\Lambda_T - \Lambda_0},
\end{align*}
where 
\begin{align*}
T \geq \frac{6 L \log(12) }{\mu_d} + \frac{6 L \log (6K^2)  }{\mu_d}  , ~ \Lambda_t = (1 + \frac{\mu_d}{6 L })^t.
\end{align*}
At the $k$-th  iteration of the $e$-th epoch, set $(\tilde{z}_k, z_k)$ by the output of \texttt{REG}($\Phi_k$), initialized at $(\hat{x}_k, y_{k-1})$ and running for a total of $T$ steps  with   
$
\eta_t =  \frac{(k+2)}{3(k+1) L} .
$
Then
\begin{align*}
 {f(\tilde{x}_{(e)}) - f(x^*)} +  \frac{\mu_d}{6} \norm{\tilde{y}_{(e)}^* - y_{(e)}}^2 
 \leq 
 \rbr{\frac{1}{2} }^e
 \sbr{
 {f(\tilde{x}_{(0)}) - f(x^*)} +  \frac{\mu_d}{6} \norm{\tilde{y}_{(0)}^* - y_{(0)}}^2 
 }.
\end{align*}
\end{theorem}

\begin{proof}
The proof follows similar lines as the proof of Theorem \ref{thrm_minimax_catalyst_deterministic_nsc}, except that we will use Lemma \ref{lemma_deterministic_minimax_sc} in place of Lemma \ref{lemma_general_convergence_minimax_deterministic_nsc}.
\end{proof}

Clearly, in view of Theorem \ref{thrm_minimax_catalyst_deterministic_nsc} and \ref{thrm_minimax_deterministic_iter_complexity}, the total iteration complexity of the proposed catalyst scheme for \eqref{def_problem} is bounded by 
$\tilde{\cO} (L/\sqrt{\mu_d \epsilon})$ (resp. $\tilde{\cO} ( L /\sqrt{\mu_d \mu_p} \log(1/\epsilon))$) 
for convex-strongly-concave (resp. strongly-convex-strongly-concave) problems. 
Notably, the established complexities here are optimal up to potential logarithmic factors. 
In the next section, we proceed to establish that the proposed catalyst scheme, when catalyzing the SREG method introduced in Section \ref{sec_eg},  is indeed capable of simultaneously achieving optimal iteration and sample complexities in the stochastic regime up to logarithmic factors.


\section{Stochastic  Minimax Catalyst Scheme}\label{sec_minimax_stoch}
The minimax catalyst scheme for the stochastic setting shares the same update as its deterministic counterpart presented in Algorithm \ref{alg_catalyst_deterministic}.
Compared to the deterministic setting in Section \ref{sec_minimax_deterministic}, our discussion in the stochastic setting requires a more refined characterization on approximately solving the proximal step \eqref{eq:cp_subproblem_minmax}.
In particular, 
we assume that the to-be-catalyzed method  can find a pair of solution $(\tilde z_k, z_k) \in X \times X$ of \eqref{eq:cp_subproblem_minmax} such that
\begin{align} \label{eq_stoch_minimax_err_condition_catalyst}
\EE \sbr{ \Phi_k(\tilde x_k, \tilde{y}) - \Phi_k(\tilde x, \tilde y_k) + \tfrac{\alpha_k}{2}  \norm{\tilde{x} - x_k}^2 + \tfrac{\alpha_k}{2}  \norm{\tilde{y} - y_k}^2 } 
\leq \EE \sbr{ \tfrac{\varepsilon_k'}{2} \|\tilde x - x_k^0 \|^2 
+ \frac{\varepsilon_k}{2} \norm{\tilde{y} - y_k^0}^2
 } + \delta_k,
\end{align}
for any potentially  random $\tilde{x}$ that is measurable with respect to the filtration generated up to iteration $k-1$, 
and random $\tilde{y}$ measurable with respect to the filtration generated up to  iteration $k$. 
Notably, $\delta_k$ here encapsulates the stochastic error when solving \eqref{eq:cp_subproblem_minmax}, and $z_k^0 \equiv (x_k^0, y_k^0)$ denotes the initial point of the subroutine for solving \eqref{eq:cp_subproblem_minmax}.
As before, we choose $x_k^0 = \hat x_k$ and $y_k^0=y_{k-1}$.

In view of  \eqref{eq_stoch_minimax_err_condition_catalyst},  
  the only difference in characterizing the approximate proximal step in the stochastic setting, compared to its deterministic counterpart \eqref{eq:cp_inexact_sd},  is the different error terms associated with the primal and dual variables.
 This is due to the fact that the catalyst scheme essentially requires a rather precise dual solution in solving the proximal step \eqref{eq:cp_subproblem_minmax}, as the dual reference point $\tilde{y}$ in  \eqref{eq_stoch_minimax_err_condition_catalyst}  can be chosen after the proximal step is solved. 
Consequently, we need to take a more refined treatment on the dual variables compared to the primal variables in solving \eqref{eq:cp_subproblem_minmax}.

In Section \ref{subsec_smeg_in_catalyst}, we will introduce a variant of the SREG method introduced in Section \ref{sec_eg} that can certify condition \eqref{eq_stoch_minimax_err_condition_catalyst} with proper parameter specifications,
and subsequently determine the sample complexity of the catalyst scheme when catalyzing the SREG method. 
Before that, we proceed to establish some generic convergence properties of the minimax catalyst scheme provided error condition \eqref{eq_stoch_minimax_err_condition_catalyst} holds.

\begin{lemma}\label{lemma_stoch_minimax_one_step}
Suppose \eqref{eq_stoch_minimax_err_condition_catalyst} holds. Then within Algorithm~\ref{alg:basic_cat_sd},
for any $k \ge 1$, we have
\begin{align}
&\rbr{1-\tfrac{4 \varepsilon_k}{\mu_d}} \EE \sbr{ f(\tilde x_k) - f(x^*)}
+ \tfrac{\alpha_k \gamma_k^2 + \gamma_k(1-\gamma_k) \mu_p}{2} \EE \sbr{ \norm{x^* -  \bar x_k}^2} +  \tfrac{\alpha_k}{2}  \EE \sbr{\norm{\tilde y_k^* - y_k}}^2 \nonumber  \\
\leq & \rbr{1-\gamma_k + \tfrac{4 \varepsilon_k}{\mu_d}} \EE \sbr{ f(\tilde x_{k-1}) - f(x^*)} +
 \tfrac{(\beta_k + \varepsilon_k') \gamma_k^2}{2} \EE \sbr{\norm{x^* - \bar x_{k-1}  }^2} + \varepsilon_k \EE \sbr{\norm{\tilde y_{k-1}^* - y_{k-1}}^2} + \delta_k  . \label{minimax_outer_one_step_recursion_stochastic}
\end{align}
\end{lemma}

\begin{proof}
The proof follows the same line as Lemma \ref{lemma_recursion_minimax_catalyst_one_step_deterministic}, by replacing \eqref{eq:cp_inexact_sd} therein by \eqref{eq_stoch_minimax_err_condition_catalyst}.
\end{proof}

The following lemma provides a basic characterization on each step of the catalyst scheme in the stochastic setting.

\begin{lemma}\label{lemma_minimax_outer_recursion_raw_stochastic}
Suppose \eqref{eq_stoch_minimax_err_condition_catalyst} holds,
and
\begin{align}
\frac{\alpha_k \gamma_k^2}{1 - 4 \varepsilon_k / \mu_d}
& \geq \frac{(\beta_{k+1} + \varepsilon_{k+1}') \gamma_{k+1}^2 }{1 - \gamma_{k+1} + 4 \varepsilon_{k+1} / \mu_d}, \label{minimax_outer_param_choice_1_stoch_1} \\
~~
 \frac{\varepsilon_{k+1}}{\alpha_k}  & \leq \frac{1 - \gamma_{k+1} + 4 \varepsilon_{k+1} / \mu_d}{2 (1 - 4 \varepsilon_k / \mu_d)}.
\label{minimax_outer_param_choice_1_stoch_2}
\end{align}
Then with $\cbr{\Gamma_k}$ defined as in \eqref{def_Gamma_minimax_deterministic}, we have 
\begin{align*}
& \EE \sbr{ {f(\tilde{x}_k) - f(x^*)} + \frac{\alpha_k}{2 (1-4 \varepsilon_k / \mu_d) } \norm{\tilde{y}_k^* - y_k}^2 } \\
\leq  &
\Gamma_k \big[ \frac{1-\gamma_1 + 4\varepsilon_1 / \mu_d}{(1 - 4 \varepsilon_1/\mu_d)}  \sbr{f({x}_0) - f(x^*)}  + \frac{(\beta_1 + \varepsilon_1') \gamma_1^2}{2 (1 - {4 \varepsilon_1}/{\mu_d}) }  \norm{x - {x}_0}^2   + 
\frac{\varepsilon_1}{(1 - {4 \varepsilon_1}/{\mu_d}) } \norm{\tilde{y}_0^* - y_0}^2 
  +
\tsum_{l=1}^k \frac{\delta_l}{(1-4 \varepsilon_l / \mu_d) \Gamma_l }
\big].
\end{align*}
\end{lemma}

\begin{proof}
The proof follows  the same line as Lemma \ref{lemma_minimax_outer_recursion_raw}, by using Lemma \ref{lemma_stoch_minimax_one_step} in the place of Lemma \ref{lemma_recursion_minimax_catalyst_one_step_deterministic}.
\end{proof}

We proceed to establish the convergence of the catalyst scheme with some concrete parameter choice and realization of \eqref{eq_stoch_minimax_err_condition_catalyst}. 
As before, we first restrict our attention to convex-strongly-concave problems ($\mu_p = 0$).

\begin{lemma}\label{lemma_general_convergence_minimax_stoch_nsc}
Suppose $\mu_p = 0$. 
Fix total iterations $K \geq 1$ a priori. 
Choose 
\begin{align*}
\gamma_k = \frac{2}{k+1},  
~ 
\beta_k =   \frac{ \mu_d (k+1)}{4(k+2)}.
\end{align*}
In addition, suppose $\alpha_k$ is chosen such that there exists  $\varepsilon_k$ certifying \eqref{eq_stoch_minimax_err_condition_catalyst} with 
\begin{align}
\alpha_k = \beta_k (1 + \varepsilon'), ~ \varepsilon_k' & = \beta_k \varepsilon', ~ \varepsilon_k = \beta_k \varepsilon,  \label{general_minimax_stoch_alpha_beta} \\
 \delta_k & \leq \delta , 
  \label{general_minimax_stoch_noise_condition}
\end{align}
for some $\delta > 0$ and
\begin{align}\label{general_minimax_stoch_varepsilon}
\varepsilon \leq \min \cbr{\frac{1}{12}, \frac{1}{(K+1) (K+2)}, \frac{\norm{x^* - \tilde{x}_0}^2}{2 \sbr{f(\tilde{x}_0) - f(x^*) }}}, ~
\varepsilon' \leq 1.
\end{align}
Then \texttt{Catalyst-Minimax} scheme (Algorithm \ref{alg_catalyst_deterministic}) satisfies  
\begin{align*}
\EE \sbr {f(\tilde{x}_K) - f(x^*)  +  \frac{\mu_d}{12} \norm{\tilde{y}_K^* - y_K}^2 }
\leq  
\frac{12}{K^2}  \sbr{   2 \mu_d \norm{x^* - \tilde{x}_0}^2   + 
\mu_d \norm{\tilde{y}_0^* - y_0}^2
}
+ 64 K \delta
.
\end{align*}
\end{lemma}

\begin{proof}
Given the choice of $\varepsilon_k$, $\beta_k$ and $\varepsilon \leq \frac{1}{(K+1)(K+2)}$,  it can be readily verified that $\frac{4 \varepsilon_{k+1}}{\mu_d} \leq \frac{2}{(k+1)(k+2)}$ for $k \leq K$.
Combining this with the definition of $\gamma_k = \frac{2}{k+1}$,  it holds that
\begin{align}\label{stoch_minimax_gamma_ratio_bound}
\frac{1 - \gamma_{k+1} + 4 \varepsilon_{k+1} / \mu_d}{1 - 4 \varepsilon_k / \mu_d} \in \sbr{ \frac{k}{k+2}, \frac{k+1}{k+3}}.
\end{align}
Since $\varepsilon_{k+1} \geq \varepsilon_k$ and $\gamma_k \geq \gamma_{k+1}$, the above relation implies
\begin{align*}
\frac{1 - \gamma_{k} + 4 \varepsilon_{k} / \mu_d}{1 - 4 \varepsilon_k / \mu_d} \leq \frac{k+1}{k+3}, ~ k \geq 1;
~
\frac{1 - \gamma_{k} + 4 \varepsilon_{k} / \mu_d}{1 - 4 \varepsilon_k / \mu_d} \geq \frac{k-1}{k+1}, ~ k \geq 2,
\end{align*}
and consequently $\Gamma_k \leq \frac{12}{k^2}$,
and $\Gamma_k \geq \frac{2}{k (k+1)}$.

It remains to verify that \eqref{minimax_outer_param_choice_1_stoch_1} and \eqref{minimax_outer_param_choice_1_stoch_2}  in Lemma \ref{lemma_minimax_outer_recursion_raw_stochastic} hold.
In particular, \eqref{minimax_outer_param_choice_1_stoch_1} is a direct consequence of \eqref{stoch_minimax_gamma_ratio_bound} and the choice of $\alpha_k = \beta_k (1 + \varepsilon')$, $\varepsilon_k' = \beta_k \varepsilon'$, $\gamma_k = \frac{2}{k+1}$, and $\beta_k = \frac{\mu_d (k+1)}{4 (k+2)}$. 
Additionally, \eqref{minimax_outer_param_choice_1_stoch_2} follows from \eqref{stoch_minimax_gamma_ratio_bound} and the choice of $\alpha_k = \beta_k (1 + \varepsilon')$, $\beta_k = \frac{\mu_d (k+1)}{4 (k+2)}$,   $\varepsilon_k = \beta_k \varepsilon$, together with $\varepsilon \leq \frac{1}{12}$.

Combining the above observations, we can now invoke Lemma \ref{lemma_minimax_outer_recursion_raw_stochastic} and obtain 
\begin{align}
&\EE  \sbr{f(\tilde{x}_K) - f(x^*)  + \frac{\mu_d}{12} \norm{\tilde{y}_K^* - y_K}^2 } \nonumber \\
\overset{(a)}{\leq} & \EE \sbr{f(\tilde{x}_K) - f(x^*) + \frac{\alpha_K}{2 (1-4 \varepsilon_K / \mu_d) } \norm{\tilde{y}_K^* - y_K}^2} \nonumber \\
\leq  &
\frac{12}{K^2} \big[ \frac{1-\gamma_1 + 4\varepsilon_1 / \mu_d}{(1 - 4 \varepsilon_1/\mu_d)}  \sbr{f({x}_0) - f(x^*)}  + \frac{(\beta_1 + \varepsilon_1') \gamma_1^2}{2 (1 - {4 \varepsilon_1}/{\mu_d}) }  \norm{x - {x}_0}^2   + 
\frac{\varepsilon_1}{(1 - {4 \varepsilon_1}/{\mu_d}) } \norm{\tilde{y}_0^* - y_0}^2 \nonumber \\
& ~~~ 
 + 
 \tsum_{l=1}^K \frac{\delta}{(1-4 \varepsilon_l / \mu_d) \Gamma_l }
\big] \nonumber \\
\overset{(b)}{\leq} & \frac{12}{K^2} \big[ \frac{1-\gamma_1 + 4\varepsilon_1 / \mu_d}{(1 - 4 \varepsilon_1/\mu_d)}  \sbr{f({x}_0) - f(x^*)}  + \frac{(\beta_1 + \varepsilon_1') \gamma_1^2}{2 (1 - {4 \varepsilon_1}/{\mu_d}) }  \norm{x - {x}_0}^2   + 
\frac{\varepsilon_1}{(1 - {4 \varepsilon_1}/{\mu_d}) } \norm{\tilde{y}_0^* - y_0}^2 \big]
+ 64 K \delta, 
\label{ineq_nonsc_stoch_deterministic_raw}
\end{align}
where $(a)$ follows from the choice of $\beta_k$ and $\varepsilon \leq \frac{1}{12}$, which subsequently implies 
\begin{align*}
\frac{\alpha_k}{2 (1 - 4 \varepsilon_k / \mu_d)} = \frac{\beta_k (1 + \varepsilon') }{2 (1 - 4 \beta_k \varepsilon / \mu_d)} \geq \frac{\mu_d}{12}.
\end{align*}
In addition, $(b)$ follows from $\Gamma_l \geq \frac{2}{k (k+1)}$ and $1 - \frac{4 \varepsilon_l }{\mu_d} \geq \frac{1}{2}$, which in turn yields 
\begin{align*}
\frac{12}{K^2} \tsum_{l=1}^K \frac{\delta}{(1-4 \varepsilon_l / \mu_d) \Gamma_l } \leq 64 K \delta. 
\end{align*}
 Now simplifying \eqref{ineq_nonsc_stoch_deterministic_raw} after plugging in the choice of $(\gamma_1, \varepsilon_1, \varepsilon_1', \beta_1)$ and applying $\varepsilon \leq \frac{1}{12}$, $\varepsilon' \leq 1$,  
 we obtain 
 \begin{align}\label{ineq_stoch_minimax_smooth_wo_cond}
&  \EE \sbr{f(\tilde{x}_K) - f(x^*) +  \frac{\mu_d}{12} \norm{\tilde{y}_K^* - y_K}^2 } \nonumber \\
\leq &   
\frac{12}{K^2}  \sbr{   2 \varepsilon  \sbr{f(\tilde{x}_0) - f(x^*)}  +  \mu_d \norm{x^* - \tilde{x}_0}^2   + 
\mu_d \norm{\tilde{y}_0^* - y_0}^2
}
+ 64 K\delta
.
\end{align}
The desired claim then follows from the choice of $\varepsilon \leq \frac{\norm{x^* - \tilde{x}_0}^2}{2 \sbr{f(\tilde{x}_0) - f(x^*) }}$.
\end{proof}

With Lemma \ref{lemma_general_convergence_minimax_stoch_nsc} in place, we can now establish the convergence of the \texttt{R-Catalyst-Minimax} scheme (Algorithm \ref{alg_restart_opt_catalyst_minimax}) in the stochastic setting.

\begin{lemma}\label{lemma_general_convergence_minimax_stoch_sc}
Suppose $\mu_p > 0$. Within each epoch of the \texttt{R-Catalyst-Minimax} scheme, choose
\begin{align*}
K_e \equiv K = 24 \sqrt{\frac{\mu_d}{\mu_p}}, ~
\gamma_k = \frac{2}{k+1},  
~ 
\beta_k =   \frac{ \mu_d (k+1)}{4(k+2)}.
\end{align*}
In addition, suppose $\alpha_k$ is chosen such that there exists  $\varepsilon_k$ certifying \eqref{eq_stoch_minimax_err_condition_catalyst} with 
\begin{align}
\alpha_k & = \beta_k (1 + \varepsilon'), ~ \varepsilon_k'  = \beta_k \varepsilon', ~ \varepsilon_k = \beta_k \varepsilon,  \label{general_minimax_stoch_alpha_beta_sc} \\
 \delta_k & \leq \delta  \coloneqq \frac{1}{128 K} \rbr{\frac{1}{2}}^e \sbr{
{f(\tilde{x}_{(0)}) - f(x^*)} +  \frac{\mu_d}{12} \norm{\tilde{y}_{(0)}^* - y_{(0)}}^2
}
  \label{general_minimax_stoch_noise_condition_sc}
\end{align}
for some $\varepsilon, \varepsilon' > 0$ satisfying 
\begin{align}\label{general_minimax_stoch_varepsilon_2}
\varepsilon \leq \min \cbr{\frac{1}{12}, \frac{1}{(K+1) (K+2)} }, ~
\varepsilon' \leq 1.
\end{align}
Then for any $e \geq 0$, we have 
\begin{align*}
 \EE \sbr{f(\tilde{x}_{(e)}) - f(x^*) +  \frac{\mu_d}{12} \norm{\tilde{y}_{(e)}^* - y_{(e)}}^2 }  
 \leq 
 \rbr{\frac{1}{2}}^e
  \EE \sbr{f(\tilde{x}_{(0)}) - f(x^*) +  \frac{\mu_d}{12} \norm{\tilde{y}_{(0)}^* - y_{(0)}}^2 }  .
\end{align*}
\end{lemma}

\begin{proof}
We proceed with an inductive argument. 
The desired claim holds trivially at $e = 0$. 
In addition, clearly \eqref{ineq_stoch_minimax_smooth_wo_cond} still holds within each epoch. 
Given the definition of $\cbr{z_{(e)}}$, this implies 
\begin{align*}
 & \EE \sbr{f(\tilde{x}_{(e+1)}) - f(x^*) +  \frac{\mu_d}{12} \norm{\tilde{y}_{(e+1)}^* - y_{(e+1)}}^2 }  \\
\leq  & 
\frac{12}{K^2}  \EE \sbr{   2 \varepsilon  \sbr{f(\tilde{x}_{(e)}) - f(x^*)}  +  \mu_d \norm{x^* - \tilde{x}_{(e)}}^2   + 
\mu_d \norm{\tilde{y}_{(e)}^* - y_{(e)}}^2
}
+ 64 K\delta \\
\overset{(a)}{\leq}  & 
\frac{12}{K^2}  \EE \sbr{   3 \frac{\mu_d}{\mu_p}  \sbr{f(\tilde{x}_{(e)}) - f(x^*)}    + 
\mu_d \norm{\tilde{y}_{(e)}^* - y_{(e)}}^2
}
+ 64 K\delta \\
\overset{(b)}{\leq} & 
\frac{1}{4}  \EE \sbr{
{f(\tilde{x}_{(e)}) - f(x^*)} +  \frac{\mu_d}{12} \norm{\tilde{y}_{(e)}^* - y_{(e)}}^2
}
+ 64 K \delta \\
\overset{(c)}{\leq} &
\rbr{\frac{1}{2}}^{(e+1)} \sbr{
{f(\tilde{x}_{(0)}) - f(x^*)} +  \frac{\mu_d}{12} \norm{\tilde{y}_{(0)}^* - y_{(0)}}^2
},
\end{align*}
where 
$(a)$ follows from $\varepsilon \leq \frac{1}{12}$ together with $f$ being $\mu_p$ strongly convex, 
$(b)$ follows from the choice of 
$
K = 24 \sqrt{\frac{\mu_d}{\mu_p}}, 
$
and $(c)$ follows from the induction hypothesis and the choice of $\delta$ in \eqref{general_minimax_stoch_noise_condition_sc}.
The proof is then completed. 
\end{proof}

Until now our discussions have been based on the presumption that \eqref{eq_stoch_minimax_err_condition_catalyst} holds with certain conditions presented in Lemma \ref{lemma_general_convergence_minimax_stoch_nsc} and \ref{lemma_general_convergence_minimax_stoch_sc} (e.g., \eqref{general_minimax_stoch_alpha_beta} and \eqref{general_minimax_stoch_noise_condition}).
In the next subsection, we introduce a slight modification of the SREG method discussed in Section~\ref{sec_eg} for solving the proximal step,
and determine the total sample complexity of the catalyst scheme when catalyzing the SREG method.

\subsection{Catalyst Scheme with Stochastic Regularized Extragradient}\label{subsec_smeg_in_catalyst}

Before we proceed to technical discussions in this subsection,  it might be worth reiterating that within the approximate proximal update \eqref{eq_stoch_minimax_err_condition_catalyst}, the dual reference point  can be chosen a posteriori (see Lemma \ref{lemma_recursion_minimax_catalyst_one_step_deterministic}). 
In the stochastic setting this creates a unique challenge as such a selection can depend on the noise when solving the proximal  step \eqref{eq:cp_subproblem_minmax}.  
To address this challenge, we will modify the SREG method discussed in Section \ref{sec_eg}, and adopt a novel analysis by introducing an auxiliary sequence defined in the dual space to handle the correlation with noise when selecting the dual reference point.

We now present a simple modification of the SREG method for solving the proximal step \eqref{eq:cp_subproblem_minmax}.
The resulting method, presented in Algorithm \ref{procedure_smeg_stoch_sc}, considers the problem of $\min_{x \in X} \max_{y \in Y} \psi(x,y) $ where $\psi$ is $L_\psi$-smooth, $\mu_X$-strongly-convex w.r.t. $x$,
and $\mu_Y$-strongly-concave w.r.t. $y$.
With a slight terminology overload, we refer to this method also as the SREG method within our discussion in this subsection. 
It can be readily seen that the only modification we make is the separation of strong-convexity/concavity modulus in defining  the ``extragradient'' step \eqref{eg_update_2_stoch} for the primal and dual variables.
As will be clarified in the ensuing Lemma \ref{lemma_seg_in_catalyst_generic}, this separation turns out to be essential for the efficiency of the catalyst scheme in the stochastic setting.

\begin{algorithm}[H]
\caption{\texttt{SREG}$(\psi)$: extragradient for $\min_{x \in X} \max_{y \in Y} \psi(x,y) $}
\begin{algorithmic}
\STATE {\bf Input:} stepsizes $\cbr{\eta_t}$, total number of steps $n > 0$, initial point $z_0 \in Z$
\FOR{t = 0, 1, \ldots, T-1}
\STATE Define ${G}(z, \xi) = [\nabla_{x} \psi(z; \xi); -\nabla_{y} \psi(z, \xi)]$, sample $\xi_t, \hat{\xi}_t$, and update
\begin{align}
\hat{z}_t &= \argmin_{z \in Z} \eta_t \inner{G(z_t, \xi_t)}{z} + \frac{1}{2} \norm{z - z_t}^2; \label{eg_update_1_stoch} \\
z_{t+1} & = \argmin_{z \in Z}\eta_t \sbr{ \inner{G(\hat{z}_t, \hat{\xi}_t)}{z} + \frac{\mu_X}{2} \norm{x - \hat{x}_t}^2 +  \frac{\mu_Y}{2} \norm{y - \hat{y}_t}^2} + \frac{1}{2} \norm{z - z_t}^2, \label{eg_update_2_stoch}
\end{align}
where $\mu_X$ (resp. $\mu_Y$) is the strong-convexity (resp. strong-concavity) modulus of $\psi$ w.r.t. $x$ (resp. $y$).
\ENDFOR
\STATE Construct $\overline{z}_T =  {\tsum_{t=0}^{T-1} \eta_t \Lambda_t \hat{z}_t }/({\tsum_{t=0}^{T-1} \eta_t \Lambda_t})$, with $\cbr{\Lambda_t}$ defined as \begin{align}
\Lambda_t = 
\begin{cases}
1,  & t = 0;  \\
(1 + \mu_X \eta_{t-1}) \Lambda_{t-1},  & t \geq 1.
\end{cases}
\label{seg_def_Lambda}
\end{align}
\STATE {\bf Output:}  $(\overline{z}_T, z_T)$.
\end{algorithmic}\label{procedure_smeg_stoch_sc}
\end{algorithm}

For notational simplicity, going forward let us denote 
$\zeta_t = G(z_t, \xi_t) - G(z_t)$ and $ \hat{\zeta}_t = G(\hat{z}_t, \hat{\xi}_t) - G(\hat{z}_t)$.
For any $\zeta \in Z$, we will often write $\zeta^X, \zeta^Y$ to denote its $X$ and $Y$ component, respectively. 
The next lemma characterizes each step of the SREG method by introducing an auxiliary sequence defined in the dual space.

\begin{lemma}\label{lemma_eg_stoch_step_characterization}
Define $u_0 = y_0$, and 
\begin{align}\label{def_dual_aux_sequence}
u_{t+1} = \argmin_{y \in Y} \eta_t {\inner{- \hat{\zeta}^Y_t}{y}  } +  \frac{\vartheta_t}{2} \norm{y - y_t}^2  + \frac{1}{2} \norm{y - u_t}^2.
\end{align}
Suppose 
$
L_\psi \eta_t \leq \frac{1}{2}.
$
Then for any $t \geq 0$,  we have
\begin{align}
 & \eta_t  \sbr{ \psi(\hat{x}_t, y) - \psi(x, \hat{y}_t)}   + \frac{\mu_{X} \eta_t + 1}{2} \norm{x - x_{t+1}}^2  + \frac{\mu_{Y} \eta_t + 1}{2} \norm{y - y_{t+1}}^2 \nonumber  \\
& ~~~  + \eta_t \inner{-\hat{\zeta}_t^Y}{u_t - \hat{y}_t} 
 + \eta_t \inner{\hat{\zeta}_t^X}{\hat{x}_t - x} 
+ \frac{1 + \vartheta_t}{2} \norm{y - u_{t+1}}^2   \nonumber  \\
\leq & \frac{1}{2} \norm{x - x_t}^2 +   \frac{1 + \vartheta_t}{2} \norm{y - y_t}^2 +  \frac{1}{2} \norm{y - u_t}^2 + 4   \eta_t^2 \rbr{ \norm{\zeta_t}^2 + \norm{\hat{\zeta}_t}^2}.
\label{seg_one_step_recursion}
\end{align}
\end{lemma}

\begin{proof}
First, the optimality condition of \eqref{eg_update_1_stoch} yields 
\begin{align}\label{eg_update_1_opt_condition_stoch}
\eta_t \inner{G(z_t, \xi_t)}{\hat{z}_t - z}
+ \frac{1}{2} \norm{\hat{z}_t - z_t}^2 + \frac{1}{2} \norm{\hat{z}_t - z}^2 
\leq \frac{1}{2} \norm{z - z_t}^2.
\end{align}
Similarly, the optimality condition of \eqref{eg_update_2_stoch} yields
\begin{align}\label{eg_update_2_opt_condition_stoch}
& \eta_t \inner{G(\hat{z}_t, \hat{\xi}_t)}{z_{t+1} - z} + \frac{\mu_{X} \eta_t + 1}{2} \norm{x - x_{t+1}}^2  + \frac{\mu_{Y} \eta_t + 1}{2} \norm{y - y_{t+1}}^2
+ \frac{1}{2} \norm{z_{t+1} - z_t}^2  \nonumber \\
\leq &
 \frac{\mu_{X} \eta_t}{2} \norm{x - \hat{x}_t}^2 +  \frac{\mu_{Y} \eta_t}{2} \norm{y - \hat{y}_t}^2 + \frac{1}{2} \norm{z - z_t}^2. 
\end{align}
By taking $z = z_{t+1}$ in \eqref{eg_update_1_opt_condition_stoch}, and combining with \eqref{eg_update_2_opt_condition_stoch}, we obtain 
\begin{align}
& \eta_t \inner{G(z_t)}{\hat{z}_t - z_{t+1}}
+ \frac{1}{2} \norm{\hat{z}_t - z_t}^2 + \frac{1}{2} \norm{\hat{z}_t - z_{t+1}}^2
+ \eta_t \inner{G(\hat{z}_t)}{z_{t+1} - z} \nonumber \\
& ~~~ + \frac{\mu_{X} \eta_t + 1}{2} \norm{x - x_{t+1}}^2  + \frac{\mu_{Y} \eta_t + 1}{2} \norm{y - y_{t+1}}^2 
 + \eta_t \inner{G(z_t, \xi_t) - G(z_t) }{\hat{z}_t - z_{t+1}} 
 \nonumber  \\
 & ~~~
 +  \eta_t \inner{G(\hat{z}_t, \hat{\xi}_t) - G(\hat{z}_t)}{z_{t+1} - \hat{z}_t}
 + \eta_t \inner{G(\hat{z}_t, \hat{\xi}_t) - G(\hat{z}_t)}{\hat{z}_t - z}
  \nonumber  \\
\leq &  \frac{\mu_{X} \eta_t}{2} \norm{x - \hat{x}_t}^2 +  \frac{\mu_{Y} \eta_t}{2} \norm{y - \hat{y}_t}^2 + \frac{1}{2} \norm{z - z_t}^2.  \nonumber
\end{align}
Then further applying Young's inequality to the above relation yields 
\begin{align}
& \eta_t \inner{G(z_t)}{\hat{z}_t - z_{t+1}}
+ \frac{1}{2} \norm{\hat{z}_t - z_t}^2 + \frac{1}{4} \norm{\hat{z}_t - z_{t+1}}^2
+ \eta_t \inner{G(\hat{z}_t)}{z_{t+1} - z} \nonumber  \\
& ~~~  + \frac{\mu_{X} \eta_t + 1}{2} \norm{x - x_{t+1}}^2  + \frac{\mu_{Y} \eta_t + 1}{2} \norm{y - y_{t+1}}^2  
 + \eta_t \inner{\hat{\zeta}_t}{\hat{z}_t - z}
  \nonumber  \\
\leq &  \frac{\mu_{X} \eta_t}{2} \norm{x - \hat{x}_t}^2 +  \frac{\mu_{Y} \eta_t}{2} \norm{y - \hat{y}_t}^2 + \frac{1}{2} \norm{z - z_t}^2 
+ 2 \eta_t^2 \rbr{ \norm{\zeta_t}^2 + \norm{\hat{\zeta}_t}^2}.
 \label{eg_recursion_raw_stoch}
\end{align}
Following similar lines as in the proof of Lemma \ref{lemma_eg_deterministic}, one also has 
\begin{align}
& \eta_t \sbr{ \inner{G(z_t)}{\hat{z}_t - z_{t+1}} +  \inner{G(\hat{z}_t)}{z_{t+1} - z}  } + \frac{1}{2} \norm{\hat{z}_t - z_t}^2 + \frac{1}{4} \norm{\hat{z}_t - z_{t+1}}^2  \nonumber \\
= & \eta_t  \sbr{ \inner{G(\hat{z}_t)}{\hat{z}_t - z} + \inner{G(z_t) - G(\hat{z}_t)}{\hat{z}_t - z_{t+1}} }  + \frac{1}{2} \norm{\hat{z}_t - z_t}^2 + \frac{1}{4} \norm{\hat{z}_t - z_{t+1}}^2  \nonumber \\
\overset{(a)}{\geq} & \eta_t \sbr{ \psi(\hat{x}_t, y) - \psi(x, \hat{y}_t) + \frac{\mu_{X}}{2} \norm{\hat{x}_t - x}^2  +  \frac{\mu_{Y}}{2} \norm{\hat{y}_t - y}^2
- L_\psi\norm{z_t - \hat{z}_t} \norm{\hat{z}_t - z_{t+1}} } \nonumber \\
&~~~ + \frac{1}{2} \norm{\hat{z}_t - z_t}^2 + \frac{1}{4} \norm{\hat{z}_t - z_{t+1}}^2 \nonumber \\
\overset{(b)}{\geq} & \eta_t  \sbr{ \psi(\hat{x}_t, y) - \psi(x, \hat{y}_t) + \frac{\mu_{X}}{2} \norm{\hat{x}_t - x}^2  +  \frac{\mu_{Y}}{2} \norm{\hat{y}_t - y}^2 }, \label{eg_extra_step_approx_stoch}
\end{align}
where $(a)$ follows from $G$ being $L_\psi$-Lipschitz, and $\psi$ is $\mu_{X}$-strongly convex w.r.t. $x$ and $\mu_Y$-strongly-concave w.r.t. $y$;  $(b)$ follows from H\"{o}lder's inequality together with $L_\psi\eta_t \leq {1}/{2}$.

From the optimality condition of \eqref{def_dual_aux_sequence}, we obtain 
\begin{align*}
& \eta_t \inner{- \hat{\zeta}_t}{u_{t+1} - y}
+ \frac{1 }{2} \norm{u_{t+1} - u_t}^2 + \frac{1 + \vartheta_t}{2} \norm{y - u_{t+1}}^2 
\leq  \frac{1}{2} \norm{y - u_t}^2 + \frac{\vartheta_t}{2} \norm{y - y_t}^2 .
\end{align*}
It follows from the above relation and the Young's inequality that 
\begin{align}\label{ineq_aux_sequence_one_step_recursion}
\eta_t \inner{-\hat{\zeta}_t}{u_t - \hat{y}_t} 
+ \eta_t \inner{- \hat{\zeta}_t}{\hat{y}_t - y}
+ \frac{1 + \vartheta_t}{2} \norm{y - u_{t+1}}^2 
\leq  \frac{1}{2} \norm{y - u_t}^2 + \frac{\vartheta_t}{2} \norm{y - y_t}^2 + \frac{\eta_t^2}{2} \norm{\hat{\zeta}_t}^2.
\end{align}
The desired claim follows by combining \eqref{eg_recursion_raw_stoch}, \eqref{eg_extra_step_approx_stoch} and \eqref{ineq_aux_sequence_one_step_recursion}.
\end{proof}

We then proceed to establish the convergence properties of the SREG method presented in Algorithm \ref{procedure_smeg_stoch_sc}. 

\begin{lemma}\label{lemma_seg_in_catalyst_generic}
Let $\cbr{u_t}$ be defined as in Lemma \ref{lemma_eg_stoch_step_characterization}. 
Suppose 
\begin{align}\label{seg_stoch_sc_condition}
4 \mu_X & \leq \mu_Y, 
\end{align}
and 
\begin{align}
L_\psi \eta_t \leq \frac{1}{2}, ~  \eta_{t-1}  \geq \eta_{t }, ~ & t \geq 1 ;  \label{seg_eta_condition} \\
 1 + \vartheta_t  = \sqrt{1 + \mu_Y \eta_{t}}, ~ & t \geq 0.  \label{seg_vartheta_condition_2}
\end{align}
Then with $\cbr{\Lambda_t}$ defined in \eqref{seg_def_Lambda}, we have
\begin{align}\label{seg_convergence_prop_general}
& \psi(\overline{x}_T, y) - \psi(x, \overline{y}_T)
+  \frac{\mu_X \Lambda_T}{2 (\Lambda_T - \Lambda_0)} \norm{x_T - x}^2  + \frac{\mu_X \Lambda_T}{2 (\Lambda_T - \Lambda_0) } \norm{y_T - y}^2
+ (\tsum_{t=0}^{T-1} \eta_t \Lambda_t)^{-1}\tsum_{t=0}^{T-1} \Lambda_t \cB_t \nonumber \\
\leq & 
\frac{\mu_X \Lambda_0}{2(\Lambda_T - \Lambda_0)} \norm{x_0 - x}^2  
+ \frac{2 \mu_X  \Lambda_0}{ \Lambda_T - \Lambda_0} \norm{y_0 - y}^2 
+ 
 \frac{4\mu_X}{ (\Lambda_T - \Lambda_0)}  \tsum_{t=0}^{T-1} \Lambda_t \eta_t^2 \rbr{\norm{\zeta_t}^2 + \norm{\hat{\zeta}_t}^2},
\end{align}
where $\cB_t = \eta_t \inner{-\hat{\zeta}_t^Y}{u_t - \hat{y}_t} 
 + \eta_t \inner{\hat{\zeta}_t^X}{\hat{x}_t - x} $.
\end{lemma}

\begin{proof}
Multiplying both sides of \eqref{seg_one_step_recursion} by $\Lambda_t$, while 
writing $\hat{\Delta}_t = \psi(\hat{x}_t, y) - \psi(x, \hat{y}_t)$, 
$d^X_t = \norm{x - x_t}^2$, $d^Y_t= \norm{y - y_t}^2$,  and $d^U_t = \norm{y - u_t}^2$, 
we obtain  
\begin{align}\label{seg_pre_telescopt_one_step}
& \eta_t \Lambda_t \hat{\Delta_t} + \frac{ (\mu_X \eta_t + 1) \Lambda_t}{2} d^X_{t+1}
+ \frac{(\mu_Y \eta_t + 1) \Lambda_t}{2} d^Y_{t+1} + \Lambda_t \cB_t 
+ \frac{(1+\vartheta_t)\Lambda_t}{2} d^U_{t+1}  \nonumber \\
\leq & 
\frac{\Lambda_t}{2} d^X_t + \frac{(1+\vartheta_t) \Lambda_t}{2} d^Y_t + \frac{\Lambda_t}{2} d^U_t 
+ 4 \Lambda_t \eta_t^2 \rbr{\norm{\zeta_t}^2 + \norm{\hat{\zeta}_t}^2} .
\end{align}
In view of  $\eta_{t-1} \geq \eta_t$ and \eqref{seg_vartheta_condition_2}, it can be verified that 
\begin{align*}
\min \cbr{
\frac{1 + \mu_Y\eta_{t-1}}{1 + \vartheta_t}, 1 + \vartheta_{t-1} 
}
\geq 
\min \cbr{
\frac{1 + \mu_Y\eta_{t-1}}{1 + \vartheta_{t-1}}, 1 + \vartheta_{t-1} 
}
= \sqrt{1 + \mu_Y \eta_{t-1}}.
\end{align*}
Consequently, by combining the above relation and  $\eta_t \leq 1/L_\psi$,  together with the fact that $\sqrt{1 + x} \geq 1 + x/ 4$ for $x \in [0,1]$, we have
\begin{align*}
\min \cbr{
\frac{1 + \mu_Y\eta_{t-1}}{1 + \vartheta_t}, 1 + \vartheta_{t-1} 
}
\geq 1 + \frac{\mu_Y \eta_{t-1} }{4} 
\geq 1 + \mu_X \eta_{t-1},
\end{align*}
where the last inequality follows from \eqref{seg_stoch_sc_condition}.
Given the definition of $\Lambda_t$ in \eqref{seg_def_Lambda}, this in turn implies
\begin{align}
(1+ \vartheta_t) \Lambda_t & \leq (1 + \mu_Y \eta_{t-1}) \Lambda_{t-1}, \label{Lambda_recursion_Y}  \\
 \Lambda_t & \leq  (1 + \vartheta_{t-1}) \Lambda_{t-1}. \nonumber
\end{align}
One can then take telescopic sum of \eqref{seg_pre_telescopt_one_step} and obtain 
\begin{align*}
& \tsum_{t=0}^{T-1} \eta_t \Lambda_t \hat{\Delta}_t + \frac{(\mu_X \eta_{T-1} + 1) \Lambda_{T-1}}{2} d^X_T
+  \frac{(\mu_Y \eta_{T-1} + 1) \Lambda_{T-1}}{2} d^Y_T
+ \tsum_{t=0}^{T-1} \Lambda_t \cB_t  \\
\leq &  \frac{\Lambda_0}{2} d^X_0
+ \frac{(1 + \vartheta_0) \Lambda_0}{2} d^Y_0 
+ \frac{\Lambda_0}{2} d^U_0 
+ 4 \tsum_{t=0}^{T-1} \Lambda_t \eta_t^2 \rbr{\norm{\zeta_t}^2 + \norm{\hat{\zeta}_t}^2} \\
 = &   \frac{\Lambda_0}{2} d^X_0
+ \frac{(2 + \vartheta_0) \Lambda_0}{2} d^Y_0 
+ 4 \tsum_{t=0}^{T-1} \Lambda_t \eta_t^2 \rbr{\norm{\zeta_t}^2 + \norm{\hat{\zeta}_t}^2} ,
\end{align*}
where the last equality follows from the definition of $u_0 = y_0$.
Dividing both sides of the above inequality by $\tsum_{t=0}^{T-1} \eta_t \Lambda_t$, and noting from \eqref{seg_def_Lambda} that 
$
\tsum_{t=0}^{T-1} \eta_t \Lambda_t = \frac{1}{\mu_X} \rbr{\Lambda_T - \Lambda_0},
$
we obtain 
\begin{align*}
& (\tsum_{t=0}^{T-1} \eta_t \Lambda_t)^{-1} \tsum_{t=0}^{T-1} \eta_t \Lambda_t \hat{\Delta}_t
+ \frac{\mu_X \Lambda_T}{2 (\Lambda_T - \Lambda_0)} d_T^X + \frac{\mu_X (\mu_Y \eta_{T-1} + 1) \Lambda_{T-1}}{2 (\Lambda_T - \Lambda_0) } d^Y_T
+ (\tsum_{t=0}^{T-1} \eta_t \Lambda_t)^{-1}\tsum_{t=0}^{T-1} \Lambda_t \cB_t \\
\leq & 
\frac{\mu_X \Lambda_0}{2(\Lambda_T - \Lambda_0)} d^X_0 
+ \frac{\mu_X (2 + \vartheta_0) \Lambda_0}{2 (\Lambda_T - \Lambda_0)} d^Y_0 
+ 
 \frac{4\mu_X}{ (\Lambda_T - \Lambda_0)}  \tsum_{t=0}^{T-1} \Lambda_t \eta_t^2 \rbr{\norm{\zeta_t}^2 + \norm{\hat{\zeta}_t}^2}.
\end{align*}
It remains to note from \eqref{Lambda_recursion_Y} that $(\mu_Y \eta_{T-1} + 1) \Lambda_{T-1} \geq (1 + \vartheta_T) \Lambda_T \geq \Lambda_T$, which yields 
\begin{align*}
& (\tsum_{t=0}^{T-1} \eta_t \Lambda_t)^{-1} \tsum_{t=0}^{T-1} \eta_t \Lambda_t \hat{\Delta}_t
+ \frac{\mu_X \Lambda_T}{2 (\Lambda_T - \Lambda_0)} d_T^X + \frac{\mu_X  \Lambda_T }{2 (\Lambda_T - \Lambda_0) } d^Y_T
+ (\tsum_{t=0}^{T-1} \eta_t \Lambda_t)^{-1}\tsum_{t=0}^{T-1} \Lambda_t \cB_t \\
\leq & 
\frac{\mu_X \Lambda_0}{2(\Lambda_T - \Lambda_0)} d^X_0 
+ \frac{\mu_X (2 + \vartheta_0) \Lambda_0}{2 (\Lambda_T - \Lambda_0)} d^Y_0 
+ 
 \frac{4\mu_X}{ (\Lambda_T - \Lambda_0)}  \tsum_{t=0}^{T-1} \Lambda_t \eta_t^2 \rbr{\norm{\zeta_t}^2 + \norm{\hat{\zeta}_t}^2}.
\end{align*}
The desired claim follows from the above relation, the definition of $\overline{z}_T$ and $\psi$ being convex-concave,
along with $2 + \vartheta_0 = 1 + \sqrt{1 + \mu_Y \eta_0} \leq 4$.
\end{proof}

With Lemma \ref{lemma_seg_in_catalyst_generic} in place, we are now ready to establish that with proper parameter choices, the SREG method in Algorithm \ref{procedure_smeg_stoch_sc} indeed produces solution satisfying \eqref{eq_stoch_minimax_err_condition_catalyst}.

\begin{lemma}\label{lemma_seg_catalyst_stoch_const_stepsize}
Fix total steps $T > 0$  a priori in the SREG method, and choose 
\begin{align*}
\eta_t = \eta \coloneqq \min \cbr{\overline{\eta}, \frac{q \log T}{\mu_X T}}, ~ 
q = \max \cbr{
\frac{2 \sbr{ \log( \overline{\mu}_X^2 D_Y^2 / \sigma^2) + \log T }}{\log T} , \frac{1}{\log T} 
},
\end{align*}
where 
$
\overline{\eta} \leq \frac{1}{2 L_\psi}, ~
\overline{\mu}_X \geq \mu_X. 
$
Suppose further that \eqref{seg_stoch_sc_condition} holds, then for any fixed $x$ and potentially randomized $y$
 that is measurable with respect to the filtration generated up to step $T$, we have
\begin{align*}
& \EE \sbr{ \psi(\overline{x}_T, y) - \psi(x, \overline{y}_T)
+  \frac{\mu_X \alpha}{2 } \norm{x_T - x}^2  + \frac{ \mu_X \alpha}{2} \norm{y_T - y}^2}  \\
\leq & \frac{\mu_X \varepsilon'}{2} \norm{x_0 - x}^2 
+ \frac{\mu_X \varepsilon}{2} \norm{y_0 - y}^2 + \delta, 
\end{align*}
where 
\begin{align}
 \varepsilon' = \frac{ \Lambda_0}{\Lambda_T - \Lambda_0}   ,  ~
\alpha  & = \frac{ \Lambda_T}{ \Lambda_T - \Lambda_0} \equiv 1 + \varepsilon', 
~ \Lambda_T = (1 + \mu_X \eta)^T,
 \label{seg_error_char_generic_1} \\
\varepsilon  & =  \frac{4  }{ {\rbr{1 + {\mu_X} \overline{\eta} }^T -1 } }  , \label{seg_error_char_generic_2} \\
~  \delta & =  \frac{16 \sigma^2 \rbr{
\log \rbr{{\overline{\mu}_X^2 D_Y^2}/{\sigma^2}} + \log T + 2
}}{\mu_X T}   . \label{seg_error_char_generic_noise}
\end{align}
In particular, 
suppose $\underline{\mu}_X \leq \mu_X$ and  $T \geq \frac{{1}}{\underline{\mu}_X \overline{\eta} }$, we have 
$
\varepsilon' \leq 1.
$
\end{lemma}

\begin{proof}
By letting $\eta_t = \eta \coloneqq \min \cbr{\overline{\eta}, \tilde{\eta}}$ for some $\tilde{\eta} > 0$ and $\overline{\eta} \in (0, \frac{1}{2 L_\psi})$, it can be readily seen that \eqref{seg_eta_condition} holds. Consequently Lemma \ref{lemma_seg_in_catalyst_generic} applies, and taking expectation of both sides in \eqref{seg_convergence_prop_general} yields 
\begin{align*}
&\EE \sbr{ \psi(\overline{x}_T, y) - \psi(x, \overline{y}_T)
+  \frac{\mu_X \Lambda_T}{2 (\Lambda_T - \Lambda_0)} \norm{x_T - x}^2  + \frac{\mu_X \Lambda_T}{2 (\Lambda_T - \Lambda_0) } \norm{y_T - y}^2
 } \nonumber \\
\leq & 
\EE \sbr { \frac{\mu_X \Lambda_0}{2(\Lambda_T - \Lambda_0)} \norm{x_0 - x}^2  
+ \frac{2 \mu_X  \Lambda_0}{ \Lambda_T - \Lambda_0} \norm{y_0 - y}^2  
 }
 + 
 \frac{8 \mu_X \sigma^2}{ (\Lambda_T - \Lambda_0)}  \tsum_{t=0}^{T-1} \Lambda_t \eta_t^2.
\end{align*}
In particular, plugging the choice of $\eta_t = \eta$ into the definition of $\Lambda_t$ in \eqref{seg_def_Lambda} yields 
\begin{align*}
& \EE \sbr{ \psi(\overline{x}_T, y) - \psi(x, \overline{y}_T)
+  \frac{ \mu_X (1 + \mu_X \eta)^T }{2 \sbr{(1 + \mu_X \eta)^T -1 } } \norm{x_T - x}^2  + \frac{ \mu_X (1 + \mu_X \eta)^T }{2 \sbr{(1 + \mu_X \eta)^T -1 } }  \norm{y_T - y}^2}  \\
\leq &\frac{ \mu_X }{2 \sbr{(1 + \mu_X \eta)^T -1 } }  \norm{x_0 - x}^2 
+ \frac{2 \mu_X }{ {(1 + \mu_X \eta)^T -1 } }  \norm{y_0 - y}^2 +  \frac{8 \mu_X \sigma^2}{  (1 + \mu_X \eta)^T - 1}  \tsum_{t=0}^{T-1} (1 + \mu_X \eta)^t \eta^2 \\
= & \frac{ \mu_X }{2 \sbr{(1 + \mu_X \eta)^T -1 } }  \norm{x_0 - x}^2 
+ \frac{2 \mu_X }{ {(1 + \mu_X \eta)^T -1 } }  \norm{y_0 - y}^2 +  8 \sigma^2 \eta \\
 \overset{(a)}{\leq} & \frac{ \mu_X }{2 \sbr{(1 + \mu_X \eta)^T -1 } }  \norm{x_0 - x}^2 
+ \frac{2 \mu_X }{ {\rbr{1 + {\mu_X} \overline{\eta} }^T -1 } }  \norm{y_0 - y}^2
+  \frac{2 \mu_X }{ {(1 + \mu_X \tilde{\eta} )^T -1 } }  \norm{y_0 - y}^2
 +  8 \sigma^2 \tilde{\eta} 
\end{align*}
where $(a)$ follows from  $\eta \leq \tilde{\eta}$.
Given the simple fact that $(1 + x)^T \geq 1 + Tx$ for $x \geq 0$ and $T \geq 1$, whenever  
\begin{align}\label{seg_in_catalyst_T_requirement}
\tilde{\eta} \geq \frac{1}{\mu_X  T},
\end{align}
we can further simplify the above relation as 
\begin{align}
& \EE \sbr{ \psi(\overline{x}_T, y) - \psi(x, \overline{y}_T)
+  \frac{ \mu_X (1 + \mu_X \eta)^T }{2 \sbr{(1 + \mu_X \eta)^T -1 } } \norm{x_T - x}^2  + \frac{ \mu_X (1 + \mu_X \eta)^T }{2 \sbr{(1 + \mu_X \eta)^T -1 } }  \norm{y_T - y}^2} \nonumber \\
\leq & \frac{ \mu_X }{2 \sbr{(1 + \mu_X \eta)^T -1 } }  \norm{x_0 - x}^2 
+ \frac{2 \mu_X }{ {\rbr{1 + {\mu_X} \overline{\eta} }^T -1 } }  \norm{y_0 - y}^2
+  \frac{8 \mu_X }{ {(1 + \mu_X \tilde{\eta} )^T  } }  \norm{y_0 - y}^2
 +  8 \sigma^2 \tilde{\eta}  \nonumber \\
\leq & \frac{ \mu_X }{2 \sbr{(1 + \mu_X \eta)^T -1 } }  \norm{x_0 - x}^2 
+ \frac{2 \mu_X }{ {\rbr{1 + {\mu_X} \overline{\eta}}^T -1 } }  \norm{y_0 - y}^2
+  \frac{8 \mu_X }{ {(1 + \mu_X \tilde{\eta} )^T  } } D_Y^2
 +  8 \sigma^2 \tilde{\eta}.
 \label{seg_constant_step_raw}
\end{align}
Now consider choosing $\tilde{\eta} = \frac{q \log T}{\mu_X T}$ for some $q > 0$.
It can be readily verified that with 
\begin{align}\label{seg_q_choice}
q = \max \cbr{
\frac{2 \sbr{ \log( \overline{\mu}_X^2 D_Y^2 / \sigma^2) + \log T }}{\log T} , \frac{1}{\log T} 
},
\end{align} 
we have \eqref{seg_in_catalyst_T_requirement} holds, and 
\begin{align*}
\frac{8 \mu_X }{ {(1 + \mu_X \tilde{\eta} )^T  } } D_Y^2
\overset{(c)}{\leq} 8 \mu_X D_Y^2 \exp \rbr{ - \frac{\mu_X \tilde{\eta} T}{2}}
= 8 \mu_X D_Y^2 T^{-q/2}  \overset{(d)}{\leq} \frac{\sigma^2}{\overline{\mu}_X T } 
\overset{(e)}{\leq} \frac{\sigma^2}{{\mu}_X T },
\end{align*}
where $(c)$ follows from the fact that $e^{x/2} \leq 1 + x$ for $x \in [0,1]$,
 $(d)$ applies the choice of $q$ in \eqref{seg_q_choice},
and $(e)$ follows from $\overline{\mu}_X \geq \mu_X$. 
Combining the above relation with \eqref{seg_constant_step_raw} and the choice of $q$ in \eqref{seg_q_choice} again, one can further obtain 
\begin{align*}
& \EE \sbr{ \psi(\overline{x}_T, y) - \psi(x, \overline{y}_T)
+  \frac{ \mu_X (1 + \mu_X \eta)^T }{2 \sbr{(1 + \mu_X \eta)^T -1 } } \norm{x_T - x}^2  + \frac{ \mu_X (1 + \mu_X \eta)^T }{2 \sbr{(1 + \mu_X \eta)^T -1 } }  \norm{y_T - y}^2}  \\
\leq & \frac{ \mu_X }{2 \sbr{(1 + \mu_X \eta)^T -1 } }  \norm{x_0 - x}^2 
+ \frac{2 \mu_X }{ {\rbr{1 + {\mu_X} \overline{\eta} }^T -1 } }  \norm{y_0 - y}^2
+ \frac{16 \sigma^2}{\mu_X T} \rbr{
\log \rbr{\frac{\overline{\mu}_X^2 D_Y^2}{\sigma^2}} + \log T + 2
}.
\end{align*}

Finally, it can be noted that whenever $T \geq \max \cbr{ \frac{1}{\underline{\mu}_X \overline{\eta}}, \frac{1}{\mu_X \tilde{\eta}}}$, we have 
\begin{align*}
(1 + \mu_X \eta)^T -1
= \min\cbr{
\rbr{ 1 + \mu_X \overline{\eta} }^T -1, (1 + \mu_X \tilde{\eta})^T - 1
}
\geq 1,
\end{align*}
again from the fact that $(1 + x)^T \geq 1 + Tx$ for $x \geq 0$ and $T \geq 1$. 
Clearly  $T \geq \frac{1}{\mu_X \tilde{\eta}}$ is already satisfied with the specified choice of $\tilde{\eta}$.
It follows from \eqref{seg_error_char_generic_1} and the above relation that $\varepsilon' \leq 1$.
\end{proof}

We are now ready to fully specify the concrete parameter choices of the catalyst scheme applied to the SREG method (Algorithm \ref{procedure_smeg_stoch_sc}),
and subsequently determine its total sample complexity.

\begin{theorem}\label{thrm_seg_nonsc_catalyst}
Suppose $\mu_p = 0$. 
For any $\epsilon > 0$, run \texttt{Catalyst-Minimax}(\texttt{SREG}) with   
\begin{align*}
K \geq \sqrt{\frac{24 \sbr{   2 \mu_d \norm{x^* - \tilde{x}_0}^2   + 
\mu_d \norm{\tilde{y}_0^* - y_0}^2
}}{\epsilon}}, ~ 
\gamma_k = \frac{2}{k+1},  
~ 
\beta_k =   \frac{ \mu_d (k+1)}{4(k+2)},
~
\alpha_k = \frac{\beta_k \Lambda_T}{\Lambda_T - \Lambda_0}.
\end{align*}
where 
\begin{align}
T & = \frac{288 L}{\mu_d} \sbr{
\log (96) + \log \rbr{48 K^2} 
+ \log \rbr{
\frac{16 \sbr{f(\tilde{x}_0) - f(x^*)}}{\norm{x^* - \tilde{x}_0}^2}
}
} \nonumber \\
& ~~~~~~ 
+ \frac{ 24756 \sigma^2 \rbr{
\log \rbr{{{\mu}_d^2 D_Y^2}/{\sigma^2}}  + 2
} K }{\epsilon {\mu}_d }
+ 
\frac{49152 \sigma^2 K }{{\mu}_d \epsilon} \max \cbr{
\log \rbr{\frac{49152 \sigma^2 K }{{\mu}_d \epsilon}}, 1} ,  \label{catalyst_with_seg_choice_of_t_nonsc} \\
\Lambda_T & = \rbr{1 + \min \cbr{\frac{\mu_d}{96 L }, \frac{q \log T}{T} }}^T. 
\end{align}
At the $k$-th  iteration of the catalyst scheme, set $(\tilde{z}_k, z_k)$ by the output of \texttt{SREG}($\Phi_k$), initialized at $(\hat{x}_k, y_{k-1})$ and running for a total of $T$ steps  with   
\begin{align}\label{seg_catalyst_stoch_nsc_stepsize}
\eta_t = \eta \coloneqq \min \cbr{\overline{\eta}, \frac{q \log T}{\beta_k T}}, ~ 
\overline{\eta} =  \frac{(k+2)}{24(k+1) L}, ~ 
q = \max \cbr{
\frac{2 \sbr{ \log( {\mu}_d^2 D_Y^2 /( 16 \sigma^2)) + \log T }}{\log T} , \frac{1}{\log T} 
}.
\end{align}
Then  
$
\EE \sbr{f(\tilde{x}_K) - f(x^*) +  \frac{\mu_d}{6} \norm{\tilde{y}_K^* - y_K}^2}  \leq \epsilon.
$
The total number of calls to SFO can be bounded by 
\begin{align*}
\cO  \bigg(
 \frac{L D_0}{\sqrt{\mu_d \epsilon}} \log \rbr{
\frac{\mu_d D_0}{\epsilon}
}  + \frac{\sigma^2 D_0 }{\epsilon^2} \rbr{ \log \rbr{\frac{\sigma^2}{\mu_d \epsilon}}  + \log \rbr{\frac{\mu_d^2 D_Y^2}{\sigma^2}} } 
\bigg),
\end{align*}
where $D_0 = \norm{x^* - \tilde{x}_0}^2   + 
 \norm{\tilde{y}_0^* - y_0}^2$.
\end{theorem}

\begin{proof}
Suppose \eqref{general_minimax_stoch_alpha_beta}, \eqref{general_minimax_stoch_noise_condition} and \eqref{general_minimax_stoch_varepsilon} hold,
then with the specified $K$, suppose further that
\begin{align}\label{seg_catalyst_concrete_param_K_delta_requirement}
64K \delta \leq \frac{\epsilon}{2},
\end{align}
one can invoke Lemma \ref{lemma_general_convergence_minimax_stoch_nsc} and obtain the desired claim as a direct consequence. 

We proceed to establish that specified parameter choices indeed satisfy \eqref{general_minimax_stoch_alpha_beta}, \eqref{general_minimax_stoch_noise_condition}, \eqref{general_minimax_stoch_varepsilon}, 
and \eqref{seg_catalyst_concrete_param_K_delta_requirement}
for each iteration of the catalyst scheme. 
Given the choice of $\beta_k$, it is clear that the subproblem \eqref{eq:def_Phi_sd} solved by the SREG method at each iteration is $\beta_k$ strongly convex in $x$ and $\mu_d + \beta_k$ strongly concave in $y$, and $2L$ smooth. 
In the context of Lemma \ref{lemma_seg_catalyst_stoch_const_stepsize} this implies that for subproblem \eqref{eq:def_Phi_sd}  we can take 
\begin{align*}
\mu_X = \beta_k, ~ \mu_Y = \mu_d + \beta_k, ~ \overline{\mu}_X = \frac{\mu_d}{4}, ~ \underline{\mu}_X = \frac{\mu_d}{6}, ~ L_\psi = 2 L .
\end{align*}
With the above specification of $(\mu_X, \mu_Y)$ and the choice of $\beta_k$ it is also clear that \eqref{seg_stoch_sc_condition} holds for subproblem \eqref{eq:def_Phi_sd}.
Note that with $\overline{\eta} = \frac{(k+2)}{24(k+1) L}$, we have $\overline{\eta} \leq \frac{1}{2L_\psi}$.
Consequently with the choice of 
\begin{align*}
\eta_t = \eta \coloneqq \min \cbr{\overline{\eta}, \frac{q \log T}{\beta_k T}}, ~ 
\overline{\eta} =  \frac{(k+2)}{24(k+1) L}, ~ 
q = \max \cbr{
\frac{2 \sbr{ \log( {\mu}_d^2 D_Y^2 /( 16 \sigma^2)) + \log T }}{\log T} , \frac{1}{\log T} 
},
\end{align*}
we can invoke Lemma \ref{lemma_seg_catalyst_stoch_const_stepsize} and obtain that \eqref{general_minimax_stoch_alpha_beta} holds with 
\begin{align}
\varepsilon' & =  \frac{ \Lambda_0}{\Lambda_T - \Lambda_0} , ~ \varepsilon = \frac{4  }{ {\rbr{1 + {\mu_X} \overline{\eta} }^T -1 } } ,  \label{eq_seg_catalyst_varepsilon_s}\\
\Lambda_T & = (1 + \beta_k \eta)^T =
\rbr{1 + \min \cbr{\frac{\mu_d}{96 L }, \frac{q \log T}{T} }}^T,
\end{align}
where the last equality holds from the choice of $\beta_k$ and $\eta$. 
Notably $\Lambda_T$ is independent of $k$.
From Lemma \ref{lemma_seg_catalyst_stoch_const_stepsize} it follows that whenever $T \geq \frac{144 L}{\mu_d} \geq  \frac{1}{\underline{\mu}_X \overline{\eta}} \equiv \frac{144 L (k+1)}{\mu_d (k+2)}$, we obtain 
\begin{align}\label{seg_catalyst_control_varepsilon_prime}
\varepsilon' \leq 1.
\end{align}
Additionally, with the choice of $\overline{\eta}$ and \eqref{eq_seg_catalyst_varepsilon_s} we have 
\begin{align*}
 \varepsilon \leq \min \cbr{\frac{1}{12}, \frac{1}{(K+1) (K+2)}, \frac{\norm{x^* - \tilde{x}_0}^2}{2 \sbr{f(\tilde{x}_0) - f(x^*) }}}
\end{align*}
whenever 
\begin{align}\label{seg_catalyst_control_varepsilon}
T \geq 
\frac{288 L}{\mu_d} \sbr{
\log (96) + \log \rbr{48 K^2} 
+ \log \rbr{
\frac{16 \sbr{f(\tilde{x}_0) - f(x^*)}}{\norm{x^* - \tilde{x}_0}^2}
}
}.
\end{align}
Combining \eqref{seg_catalyst_control_varepsilon_prime} and \eqref{seg_catalyst_control_varepsilon} shows that \eqref{general_minimax_stoch_varepsilon} is satisfied.  

It remains to show that \eqref{general_minimax_stoch_noise_condition} is satisfied. 
Given \eqref{seg_error_char_generic_noise} in Lemma \ref{lemma_seg_catalyst_stoch_const_stepsize}, it suffices to take 
\begin{align*}
T \geq 
\frac{4096 \sigma^2 K \rbr{
\log \rbr{{\overline{\mu}_X^2 D_Y^2}/{\sigma^2}}  + 2
}  }{\epsilon \underline{\mu}_X }
+ 
\frac{8192 \sigma^2 K }{\underline{\mu}_X \epsilon} \max \cbr{
\log \rbr{\frac{8192 \sigma^2 K }{\underline{\mu}_X \epsilon}}, 1} .
\end{align*}
With the specified choice of $\overline{\mu}_X$ and $\underline{\mu}_X$, the above condition can be satisfied by taking 
\begin{align}
T \geq 
\frac{ 24756 \sigma^2 \rbr{
\log \rbr{{{\mu}_d^2 D_Y^2}/{\sigma^2}}  + 2
} K }{\epsilon {\mu}_d }
+ 
\frac{49152 \sigma^2 K }{{\mu}_d \epsilon} \max \cbr{
\log \rbr{\frac{49152 \sigma^2 K }{{\mu}_d \epsilon}}, 1} .
\label{seg_catalyst_control_delta}
\end{align}
Combining \eqref{seg_catalyst_control_varepsilon} and \eqref{seg_catalyst_control_delta} yields our desired choice of $T$ specified in \eqref{catalyst_with_seg_choice_of_t_nonsc}.
Consequently, the total number of samples required by catalyst scheme is given by 
\begin{align*}
& T \cdot K  \\
= &  \cO \rbr{
\frac{ \sigma^2 \rbr{
\log \rbr{{{\mu}_d^2 D_Y^2}/{\sigma^2}}  + 2
} K^2 }{\epsilon {\mu}_d }
+ 
\frac{ \sigma^2 K^2}{{\mu}_d \epsilon} \max \cbr{
\log \rbr{\frac{ \sigma^2 }{{\mu}_d \epsilon}}, 1} 
+ 
\frac{ L K }{\mu_d} \sbr{
\log \rbr{ K^2} 
+ \log \rbr{
\frac{ \sbr{f(\tilde{x}_0) - f(x^*)}}{\norm{x^* - \tilde{x}_0}^2}
}
}
} \\
= &  
\cO ~ \bigg(
 \frac{L (\norm{x^* - \tilde{x}_0}^2   + 
 \norm{\tilde{y}_0^* - y_0}^2)}{\sqrt{\mu_d \epsilon}} \log \rbr{
\frac{\mu_d (\norm{x^* - \tilde{x}_0}^2   + 
 \norm{\tilde{y}_0^* - y_0}^2)}{\epsilon}
} \\
& ~~~ ~~~ ~~~
+ \frac{\sigma^2 (\norm{x^* - \tilde{x}_0}^2   + 
 \norm{\tilde{y}_0^* - y_0}^2) }{\epsilon^2} \rbr{ \log \rbr{\frac{\sigma^2}{\mu_d \epsilon}}  + \log \rbr{\frac{\mu_d^2 D_Y^2}{\sigma^2}} } 
\bigg),
\end{align*}
with the last equality followed from the choice of $K$.
The proof is then completed.
\end{proof}

As the last result in this section, we establish the total sample complexity of the catalyst scheme when catalyzing the SREG method for strongly-convex-strongly-concave problems. 

\begin{theorem}\label{thrm_seg_sc_catalyst}
Suppose $\mu_p > 0$. 
Run \texttt{R-Catalyst-Minimax}(\texttt{SREG}) for a total of 
\begin{align*}
E = \log_2 \rbr{
\frac{\Delta_0}{\epsilon}
} ,
 ~ \Delta_0 \coloneqq f(\tilde{x}_{(0)} ) - f(x^*) +\frac{\mu_d}{12} \norm{\tilde{y}_{(0)}^* - y_{(0)}}^2 .
\end{align*}
epochs, with epoch length $K = 24 \sqrt{\frac{\mu_d}{\mu_p}}$.
Within each epoch, choose 
\begin{align*}
\gamma_k = \frac{2}{k+1},  
~ 
\beta_k =   \frac{ \mu_d (k+1)}{4(k+2)},
~
\alpha_{k,e} = \frac{\beta_k \Lambda_{T_e}}{\Lambda_{T_e} - \Lambda_0}.
\end{align*}
where 
\begin{align}
T_e  = \frac{288 L \rbr{
\log (96) + \log \rbr{48 K^2} 
}}{\mu_d}  
 +
\frac{128 K 2^e}{\Delta_0} \bigg[
& \frac{192 \sigma^2 \rbr{ \log(\mu_d^2 D_Y^2 / \sigma^2) + 2}}{\mu_d}  + 
\frac{384 \sigma^2 \max \cbr{\log \rbr{\frac{49152 K \sigma^2}{\Delta_0}} + e \log 2, 1}}{\mu_d} 
\bigg]
 ,  \nonumber \\
\Lambda_{T_e} & = \rbr{1 + \min \cbr{\frac{\mu_d}{96 L }, \frac{q \log T_e}{T_e} }}^{T_e}. \nonumber 
\end{align}
At the $k$-th  iteration of the $e$-th epoch, set $(\tilde{z}_k, z_k)$ by the output of \texttt{SREG}($\Phi_k$), initialized at $(\hat{x}_k, y_{k-1})$ and running for a total of $T$ steps with 
\begin{align*}
\eta_t = \eta \coloneqq \min \cbr{\overline{\eta}, \frac{q \log T_e}{\beta_k T_e}}, ~ 
\overline{\eta} =  \frac{(k+2)}{24(k+1) L}, ~ 
q = \max \cbr{
\frac{2 \sbr{ \log( {\mu}_d^2 D_Y^2 /( 16 \sigma^2)) + \log T_e }}{\log T_e} , \frac{1}{\log T_e} 
}.
\end{align*}
Then we obtain 
\begin{align}\label{thrm_eq_minimax_catalyst_stoch_opt_gap}
\EE \sbr{f(\tilde{x}_{(e)}) - f(x^*) +  \frac{\mu_d}{6} \norm{\tilde{y}_{(e)}^* - y_{(e)}}^2}  \leq \epsilon.
\end{align}
The total number of calls to SFO can be bounded by 
\begin{align*}
\cO \rbr{
\frac{L}{\sqrt{\mu_d \mu_p}} \log\rbr{\frac{\mu_d}{\mu_p}}  \log_2 \rbr{\frac{\Delta_0}{\epsilon}} 
+ \frac{\sigma^2}{\mu_p \epsilon} \sbr{
\log \rbr{\frac{\mu_d^2 D_Y^2}{\sigma^2}} + 2 + \log \rbr{\frac{\sigma^2 \mu_d}{\Delta_0  \mu_p} }
}
+ \frac{\sigma^2}{\mu_p \epsilon} \log_2 \rbr{ \frac{\Delta_0}{\epsilon} }
}.
\end{align*}
\end{theorem}

\begin{proof}
Following similar lines as in Theorem \ref{thrm_seg_nonsc_catalyst}, it can be readily verified that conditions in Lemma \ref{lemma_general_convergence_minimax_stoch_sc} are satisfied by the choice of the specified parameters.
Hence \eqref{thrm_eq_minimax_catalyst_stoch_opt_gap} follows as an immediate consequence.
In addition, the total number of calls to SFO is bounded by 
\begin{align*}
& K \tsum_{e=1}^E T_e  \\
= &  \frac{288 KE L  \rbr{
\log (96) + \log \rbr{48 K^2} 
}}{\mu_d}  
+ 
\frac{128 K^2}{\Delta_0} \sbr{
\frac{192 \sigma^2 \rbr{ \log(\mu_d^2 D_Y^2 / \sigma^2) + 2}}{\mu_d}
+ \frac{384 \sigma^2 (\log(49152 K \sigma^2/\Delta_0) + 1) }{\mu_d}
}
\tsum_{e=1}^E 2^e \\
& ~~~ + \frac{49152 \log 2  K^2 \sigma^2 }{\Delta_0 \mu_d} \tsum_{e=1}^E e 2^e \\
= & \cO \rbr{
\frac{L}{\sqrt{\mu_d \mu_p}} \log\rbr{\frac{\mu_d}{\mu_p}}  \log_2 \rbr{\frac{\Delta_0}{\epsilon}} 
+ \frac{\sigma^2}{\mu_p \epsilon} \sbr{
\log \rbr{\frac{\mu_d^2 D_Y^2}{\sigma^2} }+ 2 + \log \rbr{\frac{\sigma^2 \mu_d}{\Delta_0  \mu_p} }
}
+ \frac{\sigma^2}{\mu_p \epsilon} \log_2 \rbr{ \frac{\Delta_0}{\epsilon} }
}.
\end{align*}
The proof is then completed. 
\end{proof}

In view of Theorem \ref{thrm_seg_nonsc_catalyst} and \ref{thrm_seg_sc_catalyst}, it can be seen that the sample complexity of the catalyst scheme applied to the SREG method can be bounded by 
\begin{align*}
\tilde{\cO} \rbr{
 \frac{L}{\sqrt{\mu_d \epsilon}} +  \frac{\sigma^2 }{\epsilon^2}
} ~~~ 
\text{
\bigg(resp. 
$\tilde{\cO} \rbr{
\frac{L}{\sqrt{\mu_d \mu_p}} \log(1/\epsilon) + \frac{\sigma^2 }{\mu \epsilon}
}
$
\bigg)
}
\end{align*}
for convex-strongly-concave (resp.  strongly-convex-strongly-concave) problems. 
In particular, these obtained sample complexities are optimal for their respective problem classes up to a potential logarithmic factor.  
To the best of our knowledge, the proposed catalyst scheme applied to the SREG method seems to obtain for the first time these sample complexities among first-order methods. 
Finally, it is also important to note here that existing lower bounds for the smooth and strongly-convex-strongly-concave problems are established for the duality gap or the distance to the optimal solutions \cite{zhang2019lower, ouyang2021lower}, while in this manuscript we have focused on the primal optimality gap of \eqref{def_problem}. Whether the logarithmic factors reported in this manuscript can be removed seems to remain open.


\section{Concluding Remarks}
This manuscript presents a catalyst scheme applicable to both smooth convex optimization and convex-concave minimax problems. 
We establish the optimal iteration and sample complexities of the catalyst scheme when accelerating a simple variant of stochastic gradient method for convex optimization. 
We further develop a novel variant of the stochastic extragradient method for solving strongly-monotone variational inequalities with optimal complexities. 
By catalyzing the proposed extragradient variant, the catalyst scheme further obtains optimal iteration and sample complexities for strong-convex-strongly-concave minimax problems, up to a potential logarithmic factor.
It remains highly rewarding to further investigate whether such a logarithmic factor is indeed unavoidable for minimax problems or could be removed with refined methods.

\section*{Acknowledgement}

The authors express their sincere appreciation to  Tianyi Lin for valuable discussions that inspire the development of this manuscript.

\bibliographystyle{unsrt}
{\small
\bibliography{references}
}


\appendix

\section{Supplementary Proofs}\label{sec_supp}

\begin{proof}[Proof of Lemma \ref{prop_sgd_generic_convergence}]
Denoting $\delta_{t} = \nabla \phi(u_t)  - g_t$, one clearly has 
$
\EE_{\xi_t} \sbr{\delta_t }= 0, ~ \EE_{\xi_t} \sbr{ \norm{\delta_t}_*^2 } \leq \sigma^2. 
$
In addition,  
\begin{align*}
\phi(u_t) & \overset{(a)}{\leq} \phi(u_{t-1}) + \inner{\nabla \phi(u_{t-1})}{u_t - u_{t-1}} + \frac{L_\phi}{2} \norm{u_t - u_{t-1}}^2  \\
& =  \phi(u_{t-1}) + \inner{g_{t-1}}{u_t - u_{t-1}} + \frac{1}{2 \eta_t}  \norm{u_t - u_{t-1}}^2
- \rbr{  \frac{1}{2 \eta_t} - \frac{L_\phi}{2}}  \norm{u_t - u_{t-1}}^2 +  \inner{\delta_{t-1}}{u_t - u_{t-1}}\\
& \overset{(b)}{\leq}  \phi(u_{t-1}) + \inner{g_{t-1}}{u -  u_{t-1}}  + \frac{1}{2 \eta_t}  \norm{u -  u_{t-1}}^2 - \frac{1}{2 \eta_t}  \norm{u -  u_t}^2 + 
\frac{\norm{\delta_{t-1}}_*^2}{2 (1/\eta_t - L_\phi)} \\
& =   \phi(u_{t-1}) + \inner{\nabla \phi(u_{t-1}) }{u -  u_{t-1}}  + \frac{1}{2 \eta_t}  \norm{u -  u_{t-1}}^2 - \frac{1}{2 \eta_t}  \norm{u -  u_t}^2 + 
\frac{\norm{\delta_{t-1}}_*^2}{2 (1/\eta_t - L_\phi)}
 - \inner{\delta_{t-1}}{u -  u_{t-1}} \\
 & \overset{(c)}{\leq} 
 \phi(u) + \rbr{\frac{1}{2\eta_t} - \frac{\mu_\phi}{2}}  \norm{u -  u_{t-1}}^2 - \frac{1}{2 \eta_t}  \norm{u -  u_t}^2 + 
\frac{\norm{\delta_{t-1}}_*^2}{2 (1/\eta_t - L_\phi)}
 - \inner{\delta_{t-1}}{u -  u_{t-1}},
\end{align*}
where $(a)$ follows from $\phi$ being $L_\phi$ smooth, $(b)$ follows from the Young's inequality together with $\eta_t < 1/L_\phi$, and $(c)$ follows from $\phi$ being $\mu_\phi$ strongly-convex. 
Define $\Delta_t \coloneqq \EE \sbr{\phi(u_t) - \phi(u)}$ and $D_t^{X} \coloneqq \EE \norm{u - u_t}^2$. 
Further taking expectation of the above inequality and using $\EE \sbr{ \inner{\delta_{t-1}}{u -  u_{t-1}} } = 0$  yield
\begin{align*}
\Delta_t \leq \rbr{\frac{1}{2\eta_t} - \frac{\mu_\phi}{2}}  D_{t-1}^X - \frac{1}{2 \eta_t} D_t^X + \frac{\sigma^2}{2(1/\eta_t - L_\phi)}.
\end{align*}
Now multiply both sides of the above relation by $\frac{2 \eta_t \mu_\phi}{\Lambda_t}$, we obtain
\begin{align*}
\frac{2 \eta_t \mu_\phi}{\Lambda_t} \Delta_t 
+ \frac{\mu_\phi}{\Lambda_t} D^X_t & \leq \frac{(1 -  \eta_t \mu_\phi) \mu_\phi }{\Lambda_t} D^X_{t-1} + \frac{\mu_\phi \eta_t \sigma^2}{\Lambda_t (1/\eta_t - L_\phi)} .
\end{align*}
  Summing up of the above inequality from $t=1$ to $T$, and making use of the definition of $\Lambda_t$, we obtain 
 \begin{align}\label{eq_sgd_sumup_raw}
 \tsum_{t=1}^T \frac{2 \eta_t \mu_\phi}{\Lambda_t} \Delta_t 
 + \frac{\mu_\phi}{\Lambda_T} D_T^X 
 \leq  \mu_\phi D_0^X + \tsum_{t=1}^T  \frac{\mu_\phi \eta_t \sigma^2}{\Lambda_t (1/\eta_t - L_\phi)}.
 \end{align}
 In addition, it is clear that 
 \begin{align*}
  \tsum_{t=1}^T \frac{2 \eta_t \mu_\phi}{\Lambda_t} 
  = \tsum_{t=1}^T 2 \rbr{\frac{1}{\Lambda_t} - \frac{1}{\Lambda_{t-1}}} = \frac{2(1 - \Lambda_T)}{\Lambda_T}.
 \end{align*}
 Combining the above observation with \eqref{eq_sgd_sumup_raw} and the definition of $\overline{u}_T$, we obtain 
 \begin{align*}
 \frac{2(1-\Lambda_T)}{\Lambda_T} \EE \sbr{\phi(\overline{u}_T) - \phi(s^*)} + \frac{\mu_\phi}{\Lambda_T} D_T^X 
 \leq  \mu_\phi D_0^X + \tsum_{t=1}^T  \frac{\mu_\phi  \eta_t\sigma^2}{\Lambda_t (1/\eta_t - L_\phi)}.
 \end{align*}
 The desired claim follows immediately after simple rearrangements. 
\end{proof}

\begin{proof}[Proof of Proposition \ref{prop_sgd_err_condition}]
Since $\eta_t \leq \frac{1}{2 L_\phi}$, Lemma \ref{prop_sgd_generic_convergence} applies and one can immediately obtain \eqref{sgd_err_condition_alpha_epsilon_delta}.
Moreover, direct computation yields  
\begin{align}
\Lambda_t = \Lambda_{t-1} \rbr{1 - \eta_t \mu_\phi}  = \frac{(t_0+1) (t_0 + 2)}{(t+t_0+1) (t+t_0 + 2)} .
\label{eq_sgd_lambda_t}
\end{align}
and 
\begin{align}\label{eta_divide_by_lambda_sgd}
\frac{\eta_t^2}{\Lambda_t (1-L_\phi \eta_t)} \overset{(a)}{\leq} \frac{2 \eta_t^2 }{\Lambda_t} \overset{(b)}{\leq} \frac{32}{\mu^2 t_0^2},
\end{align}
where $(a)$ follows from $\eta_t \leq \frac{1}{2L_\phi}$, 
and $(b)$ follows from \eqref{eq_sgd_lambda_t} and the definition of  $\eta_t$.
From \eqref{eq_sgd_lambda_t} and the definition of $t_0$,  
it is also clear that when $T \geq t_0$, $\Lambda_T \leq 1/2$. 
 Combining  this observation with \eqref{eq_sgd_lambda_t} and \eqref{eta_divide_by_lambda_sgd}, we obtain   
\begin{align*}
\varepsilon & =  \frac{\Lambda_T  }{1- \Lambda_T} \leq 1;  \\
\delta & =  \frac{\Lambda_T } {2(1-\Lambda_T)} \tsum_{t=1}^T  \frac{\mu_\phi \eta_t^2 \sigma^2}{\Lambda_t (1 - L_\phi \eta_t)}
\leq \frac{t_0^2}{T^2} \cdot \frac{32 \sigma^2 T}{\mu_\phi t_0^2} 
\leq \frac{32 \sigma^2}{\mu_\phi T}.
\end{align*}
The above relations immediately imply \eqref{sgd_err_condition_epsilon_delta_ub} and thus conclude the proof.
\end{proof}

\end{document}